\documentclass[]{article}
\usepackage[english]{babel}
\usepackage[normalem]{ulem}
\usepackage[all]{xy}
\usepackage{amsmath,amsthm,booktabs,color,dsfont,epsfig,graphicx,mathrsfs,multirow,scalefnt}
\usepackage{amsfonts, amssymb}
\usepackage{enumerate}

\newtheorem{theo}{Theorem}[section]
\newtheorem{cor}[theo]{Corollary}
\newtheorem{lem}[theo]{Lemma}

\newtheorem{prop}[theo]{Proposition}
\newtheorem{defn}[theo]{Definition}
\newtheorem{rmk}[theo]{Remark}
\newtheorem{ex}[theo]{Example}
\newtheorem{prob}{Problem}
\newtheorem*{conj}{Conjecture}
\newcommand{\N}{\mathbb{N}}

\newcommand{\s}{\sigma}

\newcommand{\V}{\mathcal{V}}
\newcommand{\VV}{\bar{\mathcal{V}}}
\newcommand{\U}{\mathcal{U}}

\newcommand{\A}{\mathcal{A}}
\newcommand{\h}{\tilde{H}}

\newcommand{\B}{\mathcal{B}}

\newcommand{\F}{\mathcal{F}}

\newcommand{\G}{\mathcal{G}}

\newcommand{\Pp}{\mathcal{P}}

\newcommand{\x}{\mathrm{x}}
\newcommand{\y}{\mathrm{y}}
\newcommand{\w}{\mathrm{w}}
\newcommand{\vv}{\mathrm{v}}
\newcommand{\uu}{\mathrm{u}}
\newcommand{\z}{\mathrm{z}}

\newcommand{\e}{{\mathbf{1}}}

\title{Blur shift spaces}

\author{
\small{Tadeu Zavistanovicz Almeida}\\
\footnotesize{UFSC -- Department of Mathematics}\\
\footnotesize{88040-900 Florian\'{o}polis - SC, Brazil}\\
\footnotesize{\texttt{zavistanovicz@gmail.com}}
\and
\small{Marcelo Sobottka}\\
\footnotesize{UFSC -- Department of Mathematics}\\
\footnotesize{88040-900 Florian\'{o}polis - SC, Brazil}\\
\footnotesize{\texttt{sobottka@mtm.ufsc.br}}
}

\date{}

\begin{document}

\maketitle

\begin{abstract} 
In this work we propose a new type of shift spaces, called blur shift spaces, where one can represent with a single symbol an entire set of infinite symbols. Such shift spaces are constructed from classical shift spaces, by choosing some sets of infinitely many symbols that will be represented by a new symbol, and then defining a convenient topology. The shift spaces we are proposing generalize ideas presented previously by Ott, Tomforde and Willis \cite{Ott_et_Al2014}, and by Gonçalves and Royer \cite{GRISU}, which were used to find  isomorphism between $C^*$-algebras associated to some classes of shift spaces.
In particular, blur shifts can be used as a compactification scheme for classical shift spaces.
\end{abstract}

{\bf Keywords:} Symbolic dynamics; Cellular automata; Compactification of shift spaces

{\bf MSC2020:} 37B10; 37B15; 37B99

% ---------------------------- INTRODUCTION ------------------------------------

\section{Introduction}

Recently it was proposed two new types of shift spaces: Ott-Tomforde-Willis shift spaces \cite{Ott_et_Al2014} and Gonçalves-Royer ultragraph shift spaces \cite{GRISU}. Inspired by $C^*$-algebras and the boundary path space theories for graphs, an Ott-Tomforde-Willis shift space is, roughly speaking, obtained by adding a single new symbol to a given infinite alphabet and then defining an appropriate topology that turns any shift space over this alphabet into a compact space. Gonçalves-Royer ultragraph shift spaces are obtained by first fixing a classical shift space generated from the walks on an ultragraph, and then finding a (possibly infinite) set of new symbols which are added to the alphabet. Then by defining the correspondent neighborhood system of sequences with those new symbols, the authors find sufficient conditions under which the shift space becomes locally compact. Both, Ott-Tomforde-Willis shifts and Gonçalves-Royer ultragraph shifts were successfully used to find correspondence between the conjugacy of a subclass of Markovian shift spaces and the isomorphism of their $C^*$-algebras.

The key idea of both, Ott-Tomforde-Willis shifts and Gonçalves-Royer ultragraph shifts, is to use new symbols to represent  some sets of infinitely many symbols: in Ott-Tomforde-Willis shifts the entire original alphabet is represented by a single new symbol, while in Gonçalves-Royer ultragraph shift spaces many symbols are introduced to represent the so called minimal infinite emitters of a given countably infinite ultragraph. In spite of Ott-Tomforde-Willis shifts and Gonçalves-Royer ultragraph shifts sharing the same paradigm, the constructions given in \cite{Ott_et_Al2014} and \cite{GRISU}, respectively,  differ in a fundamental way: every Ott-Tomforde-Willis shift have the same construction which consists in first adding a new symbol to the alphabet and then defining its neighborhood, while each Gonçalves-Royer ultragraph shift space is ad-hoc constructed from a previously given classical ultragraph shift space.
 
One possible interpretation for the constructions of Ott-Tomforde-Willis shifts and Gonçalves-Royer ultragraph shifts is that the new symbols added to the alphabet are representing a failure in determining what symbol would follow in some sequence when there were infinitely many ones that could follow. This interpretation leads us to propose a construction that encompasses Ott-Tomforde-Willis shifts and Gonçalves-Royer ultragraph shifts constructions. The shift spaces we are proposing here will be called blur shift spaces, alluding to the fact that some subsets of infinitely many symbols can be seen as being blurred in a way that we cannot individualize what symbol would follow in some sequence. Differently than occur in Ott-Tomforde-Willis shifts and Gonçalves-Royer ultragraph shifts, we are free to choose the subsets of symbols that will be {\em blurred} in the alphabet, and so we can turn any classical shift space into a compact or locally compact new shift space.\\

In the subsections below we shall present the background needed to develop the blur shift spaces and the formal constructions of Ott-Tomforde-Willis and Gonçalves-Royer ultragraph shift spaces. In section \ref{Blur_shift}, inspired by the schemes proposed in \cite{GRISU} and \cite{Ott_et_Al2014}, we propose a general scheme for constructing shift spaces where is considered the possibility of representing undetermined symbols. In Section \ref{sec:topology} we present a zero-dimensional topology for blur shifts and its properties, with particular focuses on criteria for its metrizability and (local) compactness. In Section \ref{sec:GSBC_for_blur_shift} we characterize continuous shift-commuting maps and generalized sliding block codes.

\subsection{Background}

Let $\N$ denote the set of all non-negative integers. Given a nonempty set $\A$ (which is called an alphabet), we define the (one-sided) full shift on $\A$ as the set

$$\A^\N:=\{(x_i)_{i\in\N}:\ x_i\in\A\ \forall i\in\N\}.$$

Given a sequence $\x=(x_i)_{i\in\N}\in \A^\N$, for each $\ell,k\in\N$ with $\ell\leq k$ we will denote the finite word $(x_\ell\ldots x_k)\in \A^{k-\ell+1}$ as $\x_{[\ell,k]}$. We endow $\A$ with the discrete topology and $\A^\N$ with the associated prodiscrete topology. We recall that a basis for the topology of $\A^\N$ are the cylinders, those clopen sets defined for each given choice of letters $a_0,\ a_1,\ \ldots,\ a_{n-1}\in\A$ as

$$[a_0a_1\ldots a_{n-1}]:=\{(x_i)_{i\in\N}:\ x_{j}=a_j\ \forall j=0,\ldots,n-1\}.$$

Note that $\A^\N$ is compact if and only if $\A$ is finite. Moreover, when $\A$ is not finite, then $\A^\N$ is not even locally compact. In any case the topology is metrizable and cylinders are always clopen sets.

Consider on $\A^\N$ the {\bf shift map} $\s:\A^\N\to\A^\N$ given by

$$\s\big((x_i)_{i\in\N}\big)=(x_{i+1})_{i\in\N}.$$

Given $\mathrm{F}\subset \bigcup_{n\geq 1} \A^n$ (which we will call the {\bf set of forbidden words}), the {\bf shift space} defined by $\mathrm{F}$ is the set $X_\mathrm{F}:=\{\x\in\A^\N: \x_{[\ell,k]}\notin\mathrm{F},\ \forall \ell,k\in\N\}$. It is well known that $\Lambda\subset\A^\N$ is a shift space (for some $\mathrm{F}\subset\bigcup_{n\geq 1}\A^n$) if and only if it is closed with respect to the topology of $\A^\N$ and $\s$-invariant (that is, $\s(\Lambda)\subset\Lambda$). Let $\Lambda\subset\A^\N$ be a shift space, then by considering on $\Lambda$ the topology induced from $\A^\N$, that is, the topology whose basis elements are the sets $[a_0a_1\ldots a_{n-1}]_\Lambda:=[a_0a_1\ldots a_{n-1}]\cap\Lambda$, and by restricting $\s$ to $\Lambda$, the pair $(\Lambda,\s)$ is a topological dynamical system.

Define $B_0(\Lambda)=\{\epsilon\}$, where $\epsilon$ represents the empty word (that is, the identity of the free group over $\A$), and for $n\geq 1$ define $B_n(\Lambda)$ as the set of all {\bf words of length $n$} that appear in some sequence of $\Lambda$, that is,

\begin{equation}\label{eq:B_n} B_n(\Lambda):=\{\x_{[0,n-1]}\in\A^n:\ \x\in\Lambda\}.\end{equation}

The {\bf language} of $\Lambda$ will be \begin{equation}\label{eq:B} B(\Lambda):=\bigcup_{n\geq 0} B_n(\Lambda).\end{equation}
It is straightforward that $B_n(\A^\N)=\A^n$ for all $n\geq 1$. Given $\uu=(u_1\ldots u_m),\vv=(v_1\ldots v_n)\in B(\A^\N)$ we define the concatenation of $\uu$ and $\vv$ as $\uu\vv=(u_1\ldots u_mv_1\ldots v_n)\in B(\A^\N)$. We also notice that $B_1(\Lambda)$ is the set of all letters of $\A$ that appear in some sequence in $\Lambda$. In particular, $B_1(\A^\N)=\A$.

Given $\w\in B(\A^\N)$ we define the {\bf follower set} of $\w$ in $\Lambda$ as the set
\begin{equation}\label{eq:followerset}\F_\Lambda(\w):=\{a\in\A:\ \w a\in B(\Lambda)\}.\end{equation}

In an analogous way, we define the {\bf predecessor set} of $\w\in B(\A^\N)$ as the set $\Pp(\w):=\{a\in\A:\ a\w \in B(\Lambda)\}$. Given $A\subset B(\A^\N)$ we will denote $$\F_\Lambda(A)=\bigcup_{\w\in A}\F_\Lambda(\w)\qquad\text{and}\qquad \Pp_\Lambda(A)=\bigcup_{\w\in A}\Pp_\Lambda(\w).$$

Note that $\F_\Lambda(\epsilon)=\Pp_\Lambda(\epsilon)=B_1(\Lambda)$, and 
$\F_\Lambda(\w)$ is empty if and only if $\w\notin B(\Lambda)$, and thus $\F_\Lambda(A)=\F_\Lambda(A\cap B(\Lambda))$ and $\Pp_\Lambda(A)=\Pp_\Lambda(A\cap B(\Lambda))$.

\subsection{Ott-Tomforde-Willis shifts and Gonçalves-Royer ultragraph shifts}\label{subsec:OTW-ultragraph}

Suppose that $\A$ is a infinite alphabet (with the discrete topology) and $\A^\N$ is the correspondent full shift (with the product topology). The most classical approach to turn a shift space $\Lambda\subset\A^\N$ compact, is through the Alexandroff compactification scheme (also known as one-point compactification - see \cite{Munkres}). The Alexandroff compactification can be applied whenever $\Lambda$ is locally compact, and it consists of adding a new point ``$\infty$'' to $\Lambda$, whose open neighbourhood system will consist of complements of compact sets. Hence we have a new space $\bar\Lambda:=\Lambda\cup\{\infty\}$ and we extend the shift map $\s:\Lambda\to\Lambda$ continuously on $\bar\Lambda$ by defining $\s(\infty)=\infty$. We could intuitively consider the point $\infty$ as a failure in determining the symbol at any position among the infinite possible symbols that could appear at any position.\\

Another way to use Alexandroff compactification consists in applying it on the infinite alphabet $\A$. In this procedure, proposed in \cite{Ott_et_Al2014}, a new point $\infty$ is added to $\A$, and its open neighbourhood system is conformed by the complement of finite sets. Thus, we have a new alphabet $\bar\A:=\A\cup\{\infty\}$ and can consider shift spaces $\Lambda\subset \bar\A^\N$ with the correspondent product topology. In this scheme, we could intuitively think a sequence $(x_i)_{i\in\N}\in\bar\A^\N$ where $x_k=\infty$ for some $k$, as a failure in determining the symbol at position $k$ among the infinite possible ones. In most of the cases it becomes natural to assume that as long as we are not able to determine the symbol in the position $k$ we will not be able to determine the symbol in any position after $k$. Under such assumption, we can consider the equivalence relation $\sim$ in $\bar\A^\N$ given by $$(x_i)_{i\in\N}\sim(y_i)_{i\in\N}\in\bar\A^\N\qquad\Leftrightarrow\qquad\min\{j: x_j=\infty\}=\min\{j: y_j=\infty\}=:k,\ and\ x_i=y_i,\ \forall\ i<k,$$
and therefore define $\Sigma_\A:=\bar\A^\N_{/\sim}$ with the quotient topology. Such compactification scheme was proposed in \cite{Ott_et_Al2014} and studied in \cite{GR,GSS}.
The set $\Sigma_\A$ can be identified with the set $\{(x_i)_{i\in\N}\in\bar\A^\N:\ x_i=\infty\Rightarrow x_{i+1}=\infty\}$. Furthermore, the quotient topology has a clopen basis whose sets were called generalized cylinders.

This compactification scheme has the advantage that we can apply it directly to $\A^\N$ (instead of being applied only for locally compact shift spaces, as the previous one). However, the shift map $\s:\Sigma_\A\to\Sigma_\A$ will be continuous at all points but at the sequence $\oslash:=(\infty,\infty,\ldots)$.

An {\bf Ott-Tomforde-Willis shift space} of $\Sigma_\A$ is the closure in $\Sigma_\A$ of a shift space $\Lambda\subset\A^\N$. Equivalently, $X\subset\Sigma_\A$ will be a shift space if and only if it is: {\em (i)} closed; {\em (ii)} shift-invariant; and {\em (iii)} for  $a_0\ldots a_n\in B(X)\cap \A^{n+1}$ we have $(a_0\ldots a_n\infty\infty\ldots)\in X$ if and only if $\F_X(a_0\ldots a_n)$ is infinite, the so called infinite-extension property.

This compactification was used to prove that two Ott-Tomforde-Willis
edge shifts over countable alphabets (an special case of Markovian shifts) will have both, the groupoids and the $C^*$-algebras of their associated graphs, being isomorphic, whenever they are topologically conjugated through a map $\Phi$ such that $\Phi^{-1}(\oslash)=\{\oslash\}$.\\

\label{ultragraph_construction} In \cite{GRISU} it was proposed a scheme to turn locally compact a class of ultragraph shifts over an countably infinite alphabet, that is, shift spaces defined from walks on a class of countably infinite ultragraphs. More specifically, given an countably infinite alphabet $\A$, let $V$ be a nonempty set (the set of vertices), $ E:=\A$ (the set of edges), $s: E\to V$ (the source function), and $r: E\to 2^{V}$ (the range function, where $2^{V}$ stands for the power set of $V$). Therefore $\mathfrak{G}=(V, E,s,r)$ will be an ultragraph, and the ultragraph shift space defined from $\mathfrak{G}$ is the shift space $X_\mathfrak{G}\subset \A^\N$ given by $$X_\mathfrak{G}:=\{(x_i)_{i\in\N}\in\A^\N:\ s(x_{i+1})\in r(x_i)\ \forall i\in\N\}.$$
 We recall that the class of ultragraph shifts coincides with the class of Markov shifts \cite[Theorem 7.3]{Sobottka2020}.

 Once $X_\mathfrak{G}$ is given, the authors in \cite{GRISU} define:\\

\noindent $\mathcal{V}_0\subset 2^{V}$ as the smallest family of sets which contains $\{v\}$ for all $v\in {V}$, contains $r(e)$ for all $e\in E$, and which is closed under finite unions and intersections;\\

\noindent $\mathcal{V}_1:=\{H\in\mathcal{V}_0:\ |s^{-1}(H)|=\infty\}$;\\

\noindent $\mathcal{\hat V}\subset\mathcal{V}_1$ as the family of all sets of $\mathcal{V}_1$ which does not contain a proper subset which also belongs to $\mathcal{V}_1$.\\

\noindent Therefore the authors define the {\bf Gonçalves-Royer ultragraph shift space}
$$\Sigma_\mathfrak{G}:= X_\mathfrak{G}\cup\{(w_i)_{i\in\N}:\ \exists \ n\in\N\ s.\ t.\ w_0\ldots w_{n-1}\in B(X_\mathfrak{G}),\ w_i= H\ \forall \ i\geq n,\ where\ H\in\mathcal{\hat V}\ and\ H\subset r(w_{n-1}) \}.$$
Then, in $\Sigma_\mathfrak{G}$ it were defined generalized cylinders in an analogous way to in \cite{Ott_et_Al2014}, which form a clopen basis for a topology on $\Sigma_\mathfrak{G}$.
Therefore, the shift map $\s:\Sigma_\mathfrak{G}\to\Sigma_\mathfrak{G}$, which is defined on $\Sigma_\mathfrak{G}$ in the usual way, is always continuous out of the set $\{(HHH\ldots):\ H\in\mathcal{\hat V}\}$.

Note that, here again, a sequence $(w_0\ldots w_{n-1}HHH\ldots)$ can be intuitively interpreted as a failure in determining the symbol which comes after $w_{n-1}$ among infinitely many possible ones. However, in this scheme it is possible to specify the subset to which the first undetermined symbol belongs (that is, $(w_0\ldots w_{n-1}HHH\ldots)$ can be thought as an indication that the symbol after $w_{n-1}$ lies in $s^{-1}(H)$).

In \cite{GRISU} it was considered the RFUM condition on the ultragraphs, which implies the local compactness of the Gonçalves-Royer ultragraph shift. Then, it was proved that two ultragraphs  $\mathfrak{G}$ and $\mathfrak{H}$ without sinks, satisfying the RFUM condition, and whose respective Gonçalves-Royer ultragraph shifts are conjugate by a map $\Phi$ such that $\hat \V_\mathfrak{G}= \bigcup_{H\in\hat\V_\mathfrak{H}}\Phi^{-1}(H)$, will have associated ultragraph $C^*$-algebras being isomorphic.\\

\section{Blur shift spaces}\label{Blur_shift}

In this section we propose a general scheme to define a new type of  shift spaces, called here blur shift spaces, where some uncertainties are represented by special symbols. We start by adding new symbols to a given alphabet which will represent some uncertainties, and then we use this extended alphabet to define a shift space with an appropriate topology. Such construction is inspired by the schemes proposed in \cite{GRISU} and \cite{Ott_et_Al2014}, and generalizes them.

\bigskip

\hrule

\subsection*{\sc Construction}

\hrule

\bigskip

Let $\A$ be an alphabet.

\begin{description}

\item[\uline{Step 1}:] Let $\V\subset 2^\A$ be any family of subsets of $\A$ such that
$$H \in \V\qquad\Rightarrow\qquad |H|=\infty$$
and 
$$G,H\in \V\text{ and } G\neq H \qquad\Rightarrow\qquad |G\cap H|<\infty.$$

The sets in $\V$ will be said to be the {\bf blurred sets} of $\A$.\\

Label each $H\in \V$ with a symbol $\h$, and denote by $\tilde\V$ the set of all symbols used to label blurred sets.

\item[\uline{Step 2}:] Let $\bar\A:=\A\cup\tilde\V$;

\item[\uline{Step 3}:] Define the full shift $\bar\A^\N$ and  consider the equivalence relation $\sim$ in $\bar\A^\N$ given by $$(x_i)_{i\in\N}\sim(y_i)_{i\in\N}\in\bar\A^\N\qquad\Leftrightarrow\qquad\min\{j: x_j\in\tilde\V\}=\min\{j: y_j\in\tilde\V\}=:k,\ and\ x_i=y_i,\ \forall\ i\leq k.$$

Define $$\Sigma_{\A^\N}^\V:=\bar\A^\N_{/\sim}.$$

\end{description}

\hrule

\bigskip

We remark that, although there is a bijection between $\V$ and $\tilde\V$,  an element in $\V$ is a subset of $\A$ while an element in $\tilde\V$ is a symbol of $\bar\A$.
Furthermore, $\h\notin H$ for any $H\in\V$. 

Given $H\in\V$ we will denote $\bar H:=H\cup \{\h\}$ which is a subset of $\bar\A$ but not of $\A$, and $\h\in\bar H$. Define $$\VV:=\{\bar H:\ H\in \V\}.$$ Note that $\VV$ is a family of subsets of $\bar \A$ which also satisfies the properties imposed in {\em Step 1} on the family $\V$.

We also remark that, if $\x\in \bar\A^\N$ is such that $x_i\in\A$ for all $i\in\N$, then $[\x]$, the equivalence class of $\x$ in $\Sigma_{\A^\N}^{\V}$ contains only $\x$. In such a case we shall identify $[\x]$ with the point $\x$ itself. On the other hand, if  $\x\in \bar\A^\N$ is such that $x_i\in\tilde\V$ for some $i\in\N$, then $[\x]$ contains infinitely many points and to represent it we will pick $(y_i)_{i\in\N}\in[\x]$ such that $y_i=x_i$ for all $i<n:=\min\{i: x_i\in\tilde\V\}$ and $y_i=x_n=\h$ for $i\geq n$.
Thus, we are going to identify $$\Sigma_{\A^\N}^{\V}\equiv\{(x_i)_{i\in\N}\in\bar\A^\N:\ x_i=\h\in \tilde\V\Rightarrow x_{i+1}=\h\}=\A^\N\cup\{(x_0...x_{n-1}\h\h\h...):\ x_0...x_{n-1}\in B(\A^\N), \h\in\tilde\V\}.$$ Hence, we can define on it the shift map  $\s:\Sigma_{\A^\N}\to\Sigma_{\A^\N}$  in the usual way. Furthermore, given $\mathfrak{X}\subset \Sigma_{\A^\N}^{\V}$, consider the sets $B_n(\mathfrak{X})$ of all words with length $n\geq 0$ in $\mathfrak{X}$, the set $B(\mathfrak{X})$ of all finite words in $\mathfrak{X}$, and the follower sets $\F_{\mathfrak{X}}(a_0\ldots a_{n-1})$ for each $a_0\ldots a_{n-1}\in B(\mathfrak{X})$, as in \eqref{eq:B_n}, \eqref{eq:B} and \eqref{eq:followerset}, respectively. We notice that a word $a_0\ldots a_{n-1}$ stands for the empty word $\epsilon$ whenever $n\leq 0$.

\begin{defn}\label{defn:blur_shift}
The space $\Sigma_{\A^\N}^{\V}$ is the {\bf full blur shift space} of $\A^\N$ with resolution $\V$. We say that $\Lambda'\subset 
\Sigma_{\A^\N}^{\V}$ is a {\bf blur shift space}  with resolution $\V$ if  and only if  there exists  a shift space $\Lambda\subset \A^\N$ such that

\begin{enumerate}
\item  
$\Lambda=\{(x_n)_{n\in \N}\in\Lambda':\ x_n\in\A \ \forall n\in\N\};$

\item $(a_0\ldots a_{n-1}\h\h\ldots)\in \Lambda'$ for some $\h\in\tilde\V\quad \Longleftrightarrow\quad a_0\ldots a_{n-1}\in B(\Lambda)$ and $|\F_{\Lambda}(a_0\ldots a_{n-1})\cap H|=\infty$.
\begin{center}
($\Lambda'$ verifies the infinite-extension property)
\end{center}

\end{enumerate}

Under the above notations, we have that $\Lambda'$ is the blur shift space of $\Lambda$ with resolution $\V$, and denote $\Lambda'=\Sigma_\Lambda^{\V}$.

\end{defn}

Note that in the above construction if $\V=\emptyset$ (which always holds if $\A$ is finite), then $\Sigma_\Lambda^{\V}=\Lambda$, which means the maximum resolution for a blur shift. On the other hand, $\V=\{\A\}$ corresponds to the minimum resolution. \\

\begin{ex}\label{ex:blurshift_count_inf_intersc} Let $\{A_n\}_{n\in\N}$ be a disjoint family of countably infinite sets, where $A_0=\N^*:=\N\setminus\{0\}$. Consider the alphabet $\A:=\bigcup_{n\geq 0}A_n$ and, for each $k\geq 1$, define the set $H_k:=A_k\cup\{i\in A_0:\ i\leq k\}$. Define $\V:=\{H_k\}_{k\geq 1}$ which is an countably infinite family of countably infinite sets. Since for all $m,n\geq 1$ with $m<n$ we have $H_m\cap H_n=\{1,...,m\}$, it follows that $\Sigma_{\A^\N}^\V$ is a full blur shift. 
\end{ex} 
% countable alphabet
% first countable
% second countable
%  \V_\Lambda countable
% metrizable according Theorem \ref{theo:Metrizable} i. or Corollary \ref{cor:2nd_count-metr}  
% neither compact nor locally compact

\begin{ex}\label{ex:blurshift_uncount_inf_intersc} Let $A_0:=\N^*$, and $\{A_n\}_{n\in\N^*}$ be a disjoint family of non-countable sets. Consider the alphabet $\A:=\bigcup_{n\geq 0}A_n$ and, for each integer $k\geq 1$, let $H_k$ be a countably infinite set such that $\{1,...,k\}\subset H_k\subset A_k\cup\{1,...,k\}$. Then, $\Sigma_{\A^\N}^\V$ with $\V:=\{H_k\}_{k\geq 1}$ is a full blur shift. 
\end{ex}  
% uncountable alphabet
% first countable
% not second countable
% \V_\Lambda countable
% metrizable according Theorem \ref{theo:Metrizable} i.
% neither compact nor locally compact

\begin{ex}\label{ex:blurshift_uncount_no_intersc} Let $\A:=\mathbb{R}^+=[0,\infty)$ and, for each $\lambda\in[0,1)$, define the set $H_\lambda:=\{x\in\mathbb{R}^+:\ x:=\lambda+k,\ k\in\N\}$. Define $\V:=\{H_\lambda\}_{\lambda\in[0,1)}$ and so the full blur shift $\Sigma_{\A^\N}^\V$.  Note that $\V$ is a non-countable family of countable disjoint sets.
\end{ex}
% uncountable alphabet 
% first countable 
% not second countable
% \V_\Lambda uncountable
% metrizable according Theorem \ref{theo:Metrizable} ii.
% neither compact nor locally compact

\begin{ex}\label{ex:prime_shift} Let $\A:=\N$ and define the sets $H_p:=\{n\in\N:\ n\text{ is prime}\}$ and $H_c:=\A\setminus H_p$.  Hence $\V:=\{H_p, H_c\}$ is a finite family of blurred sets, and $\Sigma_{\A^\N}^\V$ is a full blur shift.
\end{ex} 
% uncountable alphabet 
% first countable 
% not second countable
% first countable
% second countable
% \V_\Lambda countable
% metrizable according Theorem \ref{theo:Metrizable} i.  
% compact

\begin{ex}\label{ex:blurshift_uncount_uncont} Let $\A:=\mathbb{R}\times\mathbb{R}$ and, for each $\lambda\in\mathbb{R}$ define the set $H_\lambda:=\{\lambda\}\times\mathbb{R}$. Thus $\V:=\{H_\lambda\}_{\lambda\in\mathbb{R}}$ is a non-countable family of non-countable disjoint sets, and $\Sigma_{\A^\N}^\V$ is a full blur shift.
\end{ex} 
% uncountable alphabet 
% not first countable 
% not second countable
% not metrizable 
% neither compact nor locally compact

\begin{ex}\label{ex:intersect-all_shift} Let $\A=\N$ and let $\{H_n\subset\A: n\in\N^*\}$ be a countable family of blurred sets. Let $H_0:=\{h_1,h_2,h_3,...\}$ be any set where $h_k\in H_k\setminus\bigcup_{i<k}H_i$, and consider $\V= \{H_n\}_{ n\geq 0}$ and the respective full blur shift $\Sigma_{\A^\N}^\V$. 
\end{ex} 
% countable alphabet 
% first countable 
% second countable
%  \V_\Lambda countable
% metrizable according Theorem \ref{theo:Metrizable} i. or Corollary \ref{cor:2nd_count-metr}  
% neither compact nor locally compact

    Recall that $B_1(\Lambda)$ is the set  of symbols of $\A$ that appear in some sequence in $\Sigma_\Lambda^\V$.
Given a blur shift space $\Sigma_\Lambda^\V$, denote \begin{equation}\V_\Lambda:=\{H\in\V: \ |B_1(\Lambda)\cap H|=\infty \},\end{equation} which is the family of all blurred sets that appear in $\Sigma_\Lambda^\V$, and denote $\tilde\V_\Lambda:=\{\h:\ H\in\V_\Lambda\}$ and $\VV_\Lambda:=\{\bar H:\ H\in \V_\Lambda\}$. Furthermore, denote $\mathcal{L}_\infty^{\V}(\Lambda):=\Lambda$, and for any $n\in\N$ let \begin{equation}\mathcal{L}_n^{\V}(\Lambda):=\{(x_i)_{i\in\N}\in\Sigma_\Lambda^\V: x_n\in\tilde\V_\Lambda \text{ and } x_{n-1}\notin\tilde\V_\Lambda \},\end{equation}
and
\begin{equation}\label{eq:Ad_Lambda}\partial^\V\Lambda:=\bigcup_{n\in\N}\mathcal{L}_n^{\V}(\Lambda).\end{equation}
Note that under this notation it follows that $\Sigma_\Lambda^\V=\bigcup_{ n\in\N\cup\{\infty\}}\mathcal{L}_n^{\V}(\Lambda)=\Lambda\cup \partial^\V\Lambda$. Furthermore,  in general $\mathcal{L}_n^{\V}(\Lambda)\subsetneq\mathcal{L}_n^{\V}(\Sigma_{\A^\N})$ for any $ n\in\N\cup\{\infty\}$ (even when $B_1(\Lambda)=\A$). Finally, recall that $B_1(\Sigma_\Lambda^\V)=B_1(\Lambda)\cup\tilde\V_\Lambda$.

For simplicity of the notation, whenever the family $\V$ used to define the blur shift space is clear or off topic, we shall omit it and denote just  $\Sigma_\Lambda$, $ \mathcal{L}_n(\Lambda)$ and $\partial\Lambda$.\\

\begin{prop}\label{prop:shift_of_a_shift}
Let $\Sigma_\Lambda\subset \Sigma_{\A^\N}$ be a blur shift space. Then
\begin{enumerate}
\item\label{prop:shift_of_a_shift_i} $\s(\Sigma_\Lambda)\subset\Sigma_{\s(\Lambda)}\subset\Sigma_\Lambda$;

\item\label{prop:shift_of_a_shift_ii} $\s(\mathcal{L}_n)=\mathcal{L}_{n-1},\ \forall n\ge 1$ and $\s(\mathcal{L}_0)=\mathcal{L}_0$.

\end{enumerate}

\end{prop}

\begin{proof}\phantom\\
\begin{enumerate}
\item It is direct that $\Sigma_{\s(\Lambda)}\subset\Sigma_\Lambda$. Note that $(y_1...y_{n-1}\h\h\h...)\in\s(\partial\Lambda)$ means that there exists $y_0\in\A$ such that $(y_0y_1...y_{n-1}\h\h\h...)\in\partial\Lambda$\sloppy, that is, such that  $|\F_\Lambda (y_0y_1...y_{n-1})\cap H|=\infty$. Hence, it follows that $|\F_{\s(\Lambda)} (y_1...y_{n-1})\cap H|=\infty$, which means that  $(y_1...y_{n-1}\h\h\h...)\in\partial\s(\Lambda)$. Thus we have proved that $\s(\partial\Lambda)\subset\partial\s(\Lambda)$, we conclude that
$\s(\Sigma_\Lambda)=\s(\Lambda\cup\partial\Lambda)=\s(\Lambda)\cup\s(\partial\Lambda)\subset \s(\Lambda)\cup\partial\s(\Lambda)=\Sigma_{\s(\Lambda)}$.

\item It is straightforward.
\end{enumerate}

\end{proof}

The next example shows that in general $\Sigma_{\s(\Lambda)}\not\subset\s(\Sigma_\Lambda)$ even when $\s(\Lambda)=\Lambda$.

\begin{ex}
Let $\A$ be any infinite alphabet, and define $\V=\{\A\}$ (the Ott-Tomforde-Willis case). Consider $\Lambda:=\{(x_i)_{i\in\N}\subset\A^\N:\ x_{i+2}=x_i\ \forall i\in\N\}$. It follows that $\Sigma_\Lambda$ is such that $\partial\Lambda=\{(\tilde\A\tilde\A\tilde\A...),\ (a\tilde\A\tilde\A\tilde\A...):\ a\in\A\}$.
Therefore, from part \ref{prop:shift_of_a_shift_ii}. of Proposition \ref{prop:shift_of_a_shift}, that $\s(\Sigma_\Lambda)=\s(\Lambda)\cup \s(\partial\Lambda)=\s(\Lambda)\cup \{(\tilde\A\tilde\A\tilde\A...)\}$.

On the other hand, since $\s(\Lambda)=\Lambda$, we have $\Sigma_{\s(\Lambda)}=\Sigma_\Lambda= \Lambda\cup \partial\Lambda= \s(\Lambda)\cup\{(\tilde\A\tilde\A\tilde\A...),\ (a\tilde\A\tilde\A\tilde\A...):\ a\in\A\}$.
\end{ex}

\subsection{Graph presentation of blur shift spaces}\label{sec:graph}

We say that $\G=(V, E,s,r,L)$ is a directed labeled graph, if $V$ and $ E$ are nonempty sets (the set of vertexes and the set of edges, respectively) $s: E\to V$, $r: E\to V$ and $L: E\to\A$ are maps (the source map, the range map and the label map, respectively). So it is said that in $\G$ there is an edge $e\in E$ labeled as $a\in\A$ from the vertex $v_s\in V$ to the vertex $v_r\in V$ if and only if $s(e)=v_s$, $r(e)=v_r$ and $L(e)=a$.

A directed labeled graph $\G$ generates a shift $\Lambda_\G$ given by:

$$\Lambda_\G=\left\{\big(L(e_i)\big)_{i\geq 0}\in\A^\N:\ (e_i)_{i\in\N}\in E^\N,\ s(e_{i+1})=r(e_i)\ \forall i\geq 0\right\}.$$

In \cite{Sobottka2020} it was proved that for any shift space $\Lambda\subset \A^\N$ there exists a (possibly infinite) labeled directed graph $\G$ which represent $\Lambda$, that is,  such that $\Lambda=\Lambda_\G$. An analogous result can be obtained for blur shift spaces:

\begin{theo}\label{theo:blurgraph} Any blur shift space $\Sigma_\Lambda^\V\subset \bar \A^\N$ can be generated by a directed labeled graph $\bar \G$.
\end{theo}

\begin{proof}

Let $\G=(V, E,s,r,L)$ be the graph given in \cite[Theorem 1.1]{Sobottka2020} such that $\Lambda=\G_\Lambda$ where:  $V:=\{V(\w):\ \w\in B(\Lambda)\}$, where $V(\w):=\{\uu\in B(\Lambda):\ \w\uu\in B(\Lambda)\}$; $E:=\{e_{V(\vv)V(\vv a)}:\ \vv\in B(\Lambda),\ a\in \A,\text{ and } \vv a\in B(\Lambda)\}$; $s(e_{V(\vv)V(\vv a)})=V(\vv)$; $r(e_{V(\vv)V(\vv a)})=V(\vv a)$; and $L:E\to \A$ given by $L(e_{V(\vv)V(\vv a)})=a$. 
Note that for all $\vv\in B(\Lambda)$ we have $\F_\Lambda(\vv)=\{a\in\A:\ a\in V(\vv)\}$ and so for all $H\in\V$ it follows that $\F_\Lambda(\vv)\cap H =  V(\vv)\cap H$.\\

We build the graph $\bar \G=(\bar V, \bar E,\bar s,\bar r,\bar L)$ from $\G$ as follows:

\begin{itemize}

\item $\bar V:=V\cup\mathcal{L}_0^\V(\Lambda)$; 

\item $\bar E:=E\cup\{e_{\tilde H\tilde H}:\ \tilde H\in\mathcal{L}_0^\V(\Lambda)\}\cup\{e_{V(\vv)\tilde H}:\ \tilde H\in\mathcal{L}_0^\V(\Lambda)\text{ and }V(\vv)\in V\text{ are such that } |V(\vv)\cap H|=\infty\}$; 

\item For all $e_{ \mathbf{r}  \mathbf{s}}\in\bar E$ let $\bar s(e_{ \mathbf{r}  \mathbf{s}}):= \mathbf{r}$, $\bar r(e_{ \mathbf{r}  \mathbf{s}}):=  \mathbf{s}$, and  $\bar L(e_{ \mathbf{r}  \mathbf{s}}):=\left\{\begin{array}{lcl}a&\text{, if}&   \mathbf{s}=V(\vv a)\\ \tilde H&\text{, if}&   \mathbf{s}=\tilde H\end{array}\right.$

\end{itemize}

Hence, it is direct that $$\Lambda_{\bar \G}=\left\{\big(\bar L(e_{\mathbf{r}_i\mathbf{s}_i})\big)_{i\geq 0}\in\bar\A^\N:\ (e_{\mathbf{r}_i\mathbf{s}_i})_{i\in\N}\in \bar E^\N,\ \bar s(e_{\mathbf{r}_{i+1}\mathbf{s}_{i+1}})=\bar r(e_{\mathbf{r}_i\mathbf{s}_i})\ \forall i\geq 0\right\}=\Sigma_\Lambda^\V.$$

\end{proof}

\section{The topology of blur shift spaces}\label{sec:topology}

We consider on $\A$ the discrete topology, and on $\bar\A$ we consider the same open sets of $\A$ plus the sets $U\subset\bar\A$ that have the property that if $\h\in U$ then $H\setminus F\subset U$ for some finite $F\subset H$. Note that a basis for the topology on $\bar\A$ is the family of all singletons of $\A$ plus the sets of the form $\bar H\setminus F$ where $\bar H\in\VV$ and $F\subset H$ is finite.

On the full shift $\bar\A^\N$ we consider the product topology $\tau_{\bar\A^\N}$. The
basic open sets of  $\tau_{\bar\A^\N}$ are the cylinders, which can be written as follows: Let $S:=\{\bar H\setminus F:\ \bar H\in\VV\text{ and } F\subset H\text{ is finite}\}$, and for given  $n\ge 0$ and $a_0,..., a_{n-1}\in \A\cup S$ define a cylinder as the set
$$[a_0a_1\ldots a_{n-1}]:=\{(x_i)_{i\in\N}\in\bar\A^\N:\ \forall j=0,\ldots,n-1, x_{j}=a_j \text{ if } a_j\in\A\text{, and } x_{j}\in a_j\text{ if } a_j\in S\}.$$

 Finally, on $\Sigma_{\A^\N}$ we define the {\bf quotient topology} denoted as $\tau_{\Sigma_{\A^\N}}$. \\

Now we give a characterization of the quotient topology $\tau_{\Sigma_{\A^\N}}$. In particular, we shall prove that the quotient topology has a basis of clopen sets called {\bf generalized cylinders}, which are the sets defined for any $w_0 \dots w_{n-1} \in B(\A^\N)$, $\bar H\in\VV$, and $F \subset H$ a finite set, as

\begin{equation}\label{eq:Type 1 GC} Z(w_0\dots w_{n-1}):=\{\x \in \Sigma_{\A^\N}: x_i=w_i, 0\leq i \leq n-1\}\end{equation}

or

\begin{equation}\label{eq:Type 2 GC} Z(w_0\dots w_{n-1} \bar H,F):=\{\x \in \Sigma_{\A^\N}: x_i=w_i, 0\leq i \leq n-1, x_{n} \in \bar H \setminus F\}.\end{equation}

Note that $x_n \in \bar H\setminus F$ means that either $x_n = \h$  or $x_n \in H\setminus F$.
For simplicity of notation and further use, given $\alpha=(\alpha_0\ldots \alpha_{n-1})\in B(\A^\N)$ and $\bar H\in\VV$ we shall denote $Z(\alpha_0\ldots \alpha_{n-1} \bar H,F)=:Z(\alpha \bar H,F)$.  We will denote a generalized cylinder as  $Z(\alpha \bar H)$ whenever $F=\emptyset$, and as $Z(\bar H, F)$ whenever $\alpha= \epsilon$, the empty word (recall that $a_0\ldots a_{n-1}=\epsilon$ whenever $n\leq 0$). Furthermore, we shall consider $Z(\epsilon)=\Sigma_{\A^\N}$.\\

The relationship between cylinders of $\bar\A^\N$ and generalized cylinders of $\Sigma_{\A^\N}$ is given by the canonical projection map $q:\bar\A^\N\to\Sigma_{\A^\N}$ which maps $\tilde \x\in\bar\A^\N$ to its equivalence class $[\tilde \x]=\x\in\Sigma_{\A^\N}$. 
It is direct that 
\begin{equation}\label{eq:quotient_cylinders}\begin{array}{lcl}
q^{-1}\big(Z(\alpha)\big)&=&[\alpha_0...\alpha_{n-1}]\\\\
&\text{and}&\\\\
q^{-1}\big(Z(\alpha \bar H,F)\big)&=&[\alpha_0...\alpha_{n-1}(\bar H\setminus F)].
\end{array}
\end{equation}

\begin{prop} For any  $\alpha\in B(\A^\N)$, $\bar H\in\VV$, finite $F\subset H$, the generalized cylinders $Z(\alpha)$ and $Z(\alpha \bar H,F)$ are clopen sets of $\tau_{\Sigma_{\A^\N}}$. % Lemma 2.9 in O-T-W.
\end{prop}

\begin{proof}

To prove that $Z(\alpha)$ and $Z(\alpha \bar H,F)$ are open sets of $\tau_{\bar\A^\N}$, 
consider the canonical projection map $q:\bar\A^\N\to\Sigma_{\A^\N}$. Recall that $U\in\tau_{\Sigma_{\A^\N}}$ if and only if $q^{-1}(U)\in\tau_{\bar\A^\N}$. Thus, from \eqref{eq:quotient_cylinders} we have that $q^{-1}\big(Z(\alpha)\big)$ and
$q^{-1}\big(Z(\alpha \bar H,F)\big)$ 
are both cylinders of $\tau_{\bar\A^\N}$. Then, $Z(\alpha)$ and $Z(\alpha \bar H,F)$ are open sets of $\tau_{\Sigma_{\A^\N}}$.\\

To prove that $Z(\alpha)$ and $Z(\alpha \bar H,F)$ are closed sets of $\tau_{\Sigma_{\A^\N}}$, we will show that their complementary are open sets. 

Given $\alpha=\alpha_0...\alpha_{n-1}\in B(\A^\N)$ and $\x\in Z(\alpha)^c$, we take $j:=min\{k: x_k\neq \alpha_k\}$. Therefore, if $x_j\in \A$, then $\x\in Z(x_0...x_j)\subset Z(\alpha)^c$, while if $x_j=\tilde G\in \tilde\V$, then  $\x\in Z(x_0...x_{j-1}\tilde G,\{\alpha_j\})\subset Z(\alpha)^c$.  Since  $\x$ is any point of 
$Z(\alpha)^c$ and the correspondent generalized cylinder $Z(x_0...x_j)$ or $Z(x_0...x_{j-1}\tilde G,\{\alpha_j\})$ is an open set of $\tau_{\Sigma_{\A^\N}}$, it means that 
$Z(\alpha)^c$ is open in $\tau_{\Sigma_{\A^\N}}$ and therefore $Z(\alpha)$ is closed.

Now, suppose $\x\in Z(\alpha \bar H,F)^c$. We have two possibilities:  $x_0...x_{n-1}\neq \alpha_0...\alpha_{n-1}$ or $x_n\notin \bar H\setminus F$.  If  $x_0...x_{n-1}\neq \alpha_0...\alpha_{n-1}$, then we have $\x\in Z(\alpha)^c\subset Z(\alpha \bar H,F)^c$ and we proceed as before to find a generalized cylinder which contains $\x$ and is contained in $Z(\alpha)^c$. If $x_n\notin \bar H\setminus F$ we have two sub cases:  $x_n\in\A$ or $x_n=\tilde G \in \tilde\V\setminus\{\tilde H\}$. If $x_n\in\A$, then $\x\in Z(x_0...x_{n-1}x_n)\subset Z(\alpha\bar H,F)^c$; If $x_n=\tilde G\neq \tilde H$, then we set $F'=H \cap G$ and it follows that $\x\in Z(\alpha \bar G,F')\subset Z(\alpha\bar H,F)^c$. Hence, we conclude that $Z(\alpha \bar H,F)^c$ is open in $\tau_{\Sigma_{\A^\N}}$, thus $Z(\alpha \bar H,F)$ is closed.

\end{proof}

The proposition below ensures that the cylinders form a basis for the quotient topology on $\Sigma_\Lambda.$

\begin{prop}\label{prop:basis_top} The collection $\mathfrak{B}:=\{Z(\alpha),\ Z(\alpha \bar H,F): \alpha \in B(\A^\N), F \subset \A \text{ finite }, \bar H\in\VV\}$ is a basis for the quotient topology $\tau_{\Sigma_{\A^\N}}$. % Theorem 2.15. em O-T-W
\end{prop}

\begin{proof} Since we already proved that $\mathfrak{B}\subset \tau_{\Sigma_{\A^\N}}$, we just need to prove that for any given $U\in \tau_{\Sigma_{\A^\N}}$ and $\x\in U$, there exists $Z\in \mathfrak{B}$ such that $\x\in Z\subset U$. We remark $q^{-1}(U)$ is an open set of  $\tau_{\bar\A^\N}$, and if $\x\in U$, then $q^{-1}(U)$ contains all points of $q^{-1}(\x)$. 
 
In the case $\x\in U\cap\A^\N$ it follows that $q^{-1}(\x)=\x$ (recall that $\x$ stands for both, the point of $\bar \A^\N$ and its equivalence class in $\Sigma_{\A^\N}$). Since $q^{-1}(U)$ is an open set, there exists a cylinder $[x_0...x_n]$ in $\bar\A^\N$ such that $\x\in[x_0...x_n]\subset q^{-1}(U)$. Therefore $\x=q(\x)\in q([x_0...x_n])=Z(x_0...x_n)\subset U$.

In the case $\x\in U\cap\partial\A^\N$, say $\x\in U\cap\mathcal{L}_n(\A^\N)$ and $x_n=\tilde G$, it follows that $q^{-1}(\x)=\{\y\in\bar\A^\N:\ y_i=x_i\ \forall i\le n\}\subset q^{-1}(U)$. Now, by contradiction suppose that for all finite $F\subset G$ we have that the cylinder in $\bar\A^\N$, $[x_0...x_{n-1}\bar G\setminus F]\cap q^{-1}(U)^c\neq\emptyset$. Hence, take any $\z^0\in [x_0...x_{n-1}\bar G]\cap q^{-1}(U)^c$, and recursively, for each $\ell\geq 1$ take an arbitrary $\z^\ell\in [x_0...x_{n-1}\bar G\setminus\{\z^0,...,\z^{\ell-1}\}]\cap q^{-1}(U)^c$. Note that the sequence $(\z^\ell)_{\ell\in\N}\in\bar\A^\N$ is such that $z_0^\ell...z_{n-1}^\ell=x_0...x_{n-1}$ for all $\ell\in\N$, and for any finite $F\subset G$, there exists $L\in\N$ such that $z_n^\ell\in G\setminus F$ for all $\ell\geq L$. It means that the first $n+1$ coordinates of $\z^\ell$ are converging to $x_0...x_{n-1}\tilde G$ as $\ell$ goes to infinity, and so $q(\z^\ell)\to \x=(x_0...x_{n-1}\tilde G \tilde G \tilde G...)$ as $\ell \to\infty$. But $\z^\ell\notin q^{-1}(U)$ for all $\ell\in\N$, and so $q(\z^\ell)\notin U$, which is in contradiction to the fact that it converges to $\x$. Hence, shall exist a finite $F\subset G$ such that $[x_0...x_{n-1}\bar G\setminus F]\subset q^{-1}(U)$, and thus $Z(x_0...x_{n-1}\bar G, F)=q\big([x_0...x_{n-1}\bar G\setminus F]\big)\subset U$.

\end{proof}

Let us recall the convergence of sequences in the topology of a blur shift. 
Suppose a sequence $(\x^n)_{n\geq 1}\in \Sigma_{\A^N}$ converges to some $\bar \x$. If $\bar\x\in \A^\N$, then 
for all $k\in \N$ there is a $N\geq 1$ such that $x^n_0...x^n_k=\bar x_0...\bar x_k$ for all $n\geq N$. On the other hand, if $\bar\x\in \partial\A^\N$, say $\bar\x=(\bar x_0...\bar x_{k-1}\h\h\h...)$, then for any finite $F\subset H$ there is a 
$N\geq 1$ such that for all $n\geq N$ we have $x^n_0...x^n_{k-1}=\bar x_0...\bar x_{k-1}$ and $x^n_k\in \bar H\setminus F$.
We remark that though, in general, blur shift spaces are not first countable (see Proposition \ref{prop:first_countable}), they are always {\bf sequential spaces} (see Proposition \ref{prop:Frechet-Urysohn_space}). Thus, sequences suffice to determine the topology of blur shifts (see \cite{Franklin64}).\\

\begin{ex}\label{ex:OTW_shift} If $\A$ is any infinite set and $\V:=\{\A\}$, then the full blur shift $\Sigma_{\A^\N}^\V$ is the full Ott-Tomforde-Willis shift $\Sigma_\A$ defined in \cite{Ott_et_Al2014}. In fact, it is direct that $\Sigma_{\A^\N}^\V=\Sigma_\A$. Furthermore 
the basis given in Propostion \ref{prop:basis_top} for the topology of $\Sigma_{\A^\N}^\V$
 coincides with the basis given in \cite[Theorem 2.15]{Ott_et_Al2014} for  Ott-Tomforde-Willis shifts.

\end{ex} 
% countable or uncountable alphabet
% compact, first countabe, second countable and metrizable iff \A is countable

\begin{ex}\label{ex:ultragraph_shift}  Consider the construction of Gonçalves-Royer ultragraph shifts presented on page \pageref{ultragraph_construction} onwards. Let $\mathfrak{G}=({V}, E,s,r)$ be an ultragraph, and $\Lambda:=X_\mathfrak{G}\subset E^\N$ be the respective classical ultragraph shift. Recall that the family $\hat\V$ contains the sets of vertexes of $\mathfrak{G}$ that will represent the new symbols used in $\Sigma_\mathfrak{G}$. 
Note that each $A\in\hat\V$ can be associated to the infinite set of symbols $H_A:=s^{-1}(A)\subset E$. Since given $A,B\in\hat\V$ such that $A\neq B$ it follows that $A\cap B$ does not belong to $\hat\V$, it implies that $H_A\cap H_B$ is finite. Hence $\V:=\{H_A\}_{A\in\hat\V}$ is a family of blurred sets and then by labeling each $H_A$ with the symbol $A$, we get $\tilde\V=\hat \V$, and  $\Sigma_\mathfrak{G}=\Sigma_{\Lambda}^\V$. To conclude that  Gonçalves-Royer ultragraph shift spaces are blur shifts just observe that the basis of the topology in $\Sigma_{\Lambda}^\V$ (Proposition \ref{prop:basis_top}) coincides with the basis given in \cite[Proposition 3.4]{GRISU} for Gonçalves-Royer ultragraph shifts.\\

Observation: We notice that due to the construction of $\hat\V$ and the fact that the ultragraph has countably many vertexes and edges, we have here $\V$ being a countable family of countable sets. 
\end{ex} 
% countable alphabet
% first countable
% second countable
%  \V_\Lambda countable
% metrizable according Theorem \ref{theo:Metrizable} i. or Corollary \ref{cor:2nd_count-metr}  
% not compact
% locally compact depending on its structure

Now, we shall prove several topological properties of blur shift spaces.

\begin{prop}\label{prop:Hausdorff} The full blur shift $\Sigma_{\A^\N}$ with the topology $\tau_{\Sigma_{\A^\N}}$ is a Hausdorff topological space. % Proposition 2.5 em O-T-W
\end{prop}

\begin{proof}

To check that $\Sigma_{\A^\N}$ is Hausdorff, note that given two distinct points $\x,\y\in\Sigma_{\A^\N}$ we have three cases: $\x,\y\in\A^\N$; $\x\in\A^\N$ and $\y\in\partial\A^\N$; or $\x,\y\in\partial\A^\N$. If $\x,\y\in\A^\N$ just find $k\in\N$ such that $x_k\neq y_k$ and consider the generalized cylinders $Z(x_0...x_k)$ and $Z(y_0...y_k)$. If 
$\x\in\A^\N$ and $\y\in\partial\A^\N$, say $\y=(y_0...y_{n-1}\h\h\h...)$, we just need to take the
generalized cylinders $Z(x_0...x_n)$ and $Z(y_0...y_{n-1}\bar H,\{x_n\})$. Finally, in the case 
$\x,\y\in\partial\A^\N$, say $\x=(x_0...x_{m-1}\tilde G \tilde G \tilde G...)$ and $\y=(y_0...y_{n-1}\h\h\h...)$, if $m\neq n$ we can, without loss of generality, suppose $m<n$ and then we consider the
generalized cylinders $Z(x_0...x_{m-1}\bar G,\{y_m\})$ and $Z(y_0...y_m)$, while if $m=n$ we take the
generalized cylinders $Z(x_0...x_{m-1}\bar G)$ and $Z(y_0...y_{n-1}\bar H,G\cap H)$.
\end{proof}

\begin{prop}\label{prop:Regular} The full blur shift $\Sigma_{\A^\N}$ with the topology $\tau_{\Sigma_{\A^\N}}$ is a regular topological space. 
\end{prop}

\begin{proof}

Let $C\subset\Sigma_{\A^\N}$ be a closed set and $\y\in\Sigma_{\A^\N}$ a point not belonging to $C$. We need to find two disjoint open sets $A,B\in \tau_{\Sigma_{\A^\N}}$ such that $\y\in A$ and $C\subset B$.\\

Note that, if $\y\in\A^\N$, then there exists $m\geq 0$ such that for all $\x\in C$, we have $\x_{[0,m]}\neq \y_{[0,m]}$. In fact, if it would not hold, that is, if for each $\ell\geq 0$ there would exist $\x^\ell\in C$ such that $\x^\ell_{[0,\ell]}=\y_{[0,\ell]}$, then we would have $\x^\ell\to \y$ as $\ell\to\infty$, which is not possible since $C$ is closed and $\y\notin C$. In this case, we define $A:=Z(\y_{[0,m]})$.\\

On the other hand, if $\y=(y_0...y_{m-1}\tilde G\tilde G\tilde G...)\in\partial\A^\N$, then it follows that the set $F:=\{g\in G:\ \exists \x\in C,\text{ with } x_0...x_{m-1}x_m=y_0...y_{m-1}g\}$ is finite. In fact, if such set was not finite, then for each $\ell\geq 1$ we could take a distinct $g^\ell\in F$ and the respective sequence  $(\x^\ell)_{\ell\geq 1}\in C$ where $x^\ell_0...x^\ell_{m-1}x^\ell_m=y_0...y_{m-1}g^\ell$. Therefore, it would follows again that $x^\ell\to \y$ which is not possible due to the fact that $C$ is closed and $\y$ is not in $C$. In this case, we define $A:=Z(y_0...y_{m-1}\bar G,F)$.\\

By considering that $\y=(y_0y_1y_2...)\in\A^\N$ or $\y=(y_0...y_{m-1}\tilde G\tilde G\tilde G...)\in\partial\A^\N$, for each $\x\in C$ we define
$$B_\x:=\left\{\begin{array}{lcl}
Z(\x_{[0,k]}) &\text{, if} & x_k\in\A\text{ and } x_k\neq y_k,\text{ for some } 0\leq k\leq m,\\\\
Z(y_0...y_{k-1} \bar H,\{y_k\}) &\text{, if} &  \x=(y_0...y_{k-1} \h\h\h...)\text{ for some } 0\leq k\le m-1,\\\\
Z(y_0...y_{m-1} \bar H,\{y_m\}) &\text{, if} &  \x=(y_0...y_{m-1} \h\h\h...)\text{ and }\y\in\A^\N,\\\\
Z(y_0...y_{m-1} \bar H,H\cap G)  &\text{, if} &  \x=(y_0...y_{m-1} \h\h\h...)\text{ and }\y\in\partial\A^\N.
\end{array}\right.
$$
Hence, defining $B:=\bigcup_{\x\in C}B_\x$ we have that $A$ and $B$ are two open sets that separate $C$ and $\x$.\\

\end{proof}

Given a blur shift space $\Sigma_\Lambda\subset\Sigma_{\A^\N}$, we define on $\Sigma_\Lambda$ the topology $\tau_{\Sigma_\Lambda}$ induced from $\tau_{\Sigma_{\A^\N}}$. Thus, the basic open sets of $\Sigma_\Lambda$ are generalized cylinders in the form $Z_\Lambda(\alpha):=Z(\alpha)\cap \Sigma_\Lambda$ and $Z_\Lambda(\alpha\bar H,F):=Z(\alpha\bar H,F)\cap \Sigma_\Lambda$ for $\alpha\in B(\Lambda)$, $\bar H\in\VV_\Lambda$ and $F\subset H$ finite. \\

The next proposition shows that blur shift spaces, with the topology of generalized cylinders, are always {\bf Fréchet-Urysohn spaces} and so they are a sequential space\footnote{We refer the reader to \cite{Franklin64} and \cite[Section 1.8]{Arkhangelskii _Pontryagin} for more details about Fréchet-Urysohn spaces and sequential spaces.}:\\
 
\begin{prop}\label{prop:Frechet-Urysohn_space} Any blur shift $\Sigma_\Lambda\subset \Sigma_{\A^\N}$ is a Fréchet-Urysohn space.
\end{prop} 

\begin{proof}
Given a $A\subset \Sigma_\Lambda$, denote its closure in the topology $\tau_{\Sigma_\Lambda}$ as $\bar A$, and define its sequential closure as the set 
$$[A]_{seq}:=\{\x\in\Sigma_\Lambda:\ \exists (\y^n)_{n\geq 0}\in A\text{ s.t. } \y^n\to\x\text{ as } n\to\infty\}.$$
We need to prove that for all $A\subset \Sigma_\Lambda$ we have $\bar A\setminus A \subset[A]_{seq}$.\\

Let $\x\in \bar A\setminus A$. We shall consider two cases separately: $\x\in \Lambda$; and $\x\in \partial\Lambda$.  If $\x\in \Lambda$, then for every $n\geq 0$ the generalized cylinder $Z_\Lambda(x_0...x_n)$ intersects $A$. Hence, we can take $\y^n\in Z_\Lambda(x_0...x_n)\cap A$, and since $Z_\Lambda(x_0...x_n)\supset Z_\Lambda(x_0...x_{n+1})$ for all $n\geq 0$, it follows that $(\y^n)_{n\geq 0}$ is a sequence in $A$ converging to $\x$. Thus $\x\in [A]_{seq}$. 

If $\x\in \partial\Lambda$, say $\x=(x_0...x_{m-1}\h\h\h...)$, then for any finite $F\subset H$ we have that the generalized cylinder $Z(x_0...x_{m-1}\bar H,F)$ intersects $A$. We will construct a sequence in $A$ as follows: Choose an arbitrary $\y^0\in Z(x_0...x_{m-1}\bar H)\cap A$; for each $n\geq 1$ we will chose any point $\y^n\in Z(x_0...x_{m-1}\bar H,\{y^0_m,...,y^{n-1}_m\})\cap A$. Observe that each $\y^n$ is such that $y^n_0...y^n_{m-1}=x_0...x_{m-1}$ and $y^n_m\in H$. Furthermore $y^s_m\neq y^t_m$ for all $s\neq t$. Thus, for any finite $F\subset H$ the generalized cylinder  $Z(x_0...x_{m-1}\bar H,F)$ contains all but a finite number of terms of $(\y^n)_{n\geq 0}$, which means that $(\y^n)_{n\geq 0}$ converges to $\x\in [A]_{seq}$.

\end{proof}

The next theorem gives an alternative definition for blur shift spaces, and its proof is direct.

\begin{theo}\label{theo:blur_shift=closure_of_a_shift}
A set $\Lambda'\subset 
\Sigma_{\A^\N}^{\V}$ is a blur shift space with resolution $\V$ if  and only if  there exists  a shift space $\Lambda\subset \A^\N$ such that $\Lambda'$ is the closure of $\Lambda$ with respect to the topology $\tau_{\Sigma_{\A^\N}^\V}$.
\end{theo}

\qed

As a direct consequence of Proposition \ref{prop:shift_of_a_shift} and Theorem \ref{theo:blur_shift=closure_of_a_shift}, we have the following corollary:

\begin{cor}\label{cor:blur_shift=shift_inv+closed}
A set $\Lambda'\subset\Sigma_{\A^\N}^{\V}$ is a blur shift space with resolution $\V$ if  and only if  $\Lambda'$ is $\s$ invariant, closed in $\Sigma_{\A^\N}^{\V}$ and verifies the infinite-extension property (Definition \ref{defn:blur_shift}.ii).
\end{cor}

\qed

We recall that there are infinitely many subfamilies of $\mathfrak{B}_{\Sigma_\Lambda}$ that are basis for $\tau_{\Sigma_\Lambda}$. For instance, we could just consider cylinders $Z_\Lambda(\alpha)$ and $Z_\Lambda(\alpha \bar H, F)$, where $\alpha\in B_{n_k}(\Lambda)$ for any infinite $\{n_k\}_{k\in\N}\subset\N$. Furthermore, if $\F_\Lambda(\alpha)\cap G$  is countable, then $F$ could be restricted to any prefixed family $\{F^i:\ F^i\text{ is finite and } F^i\nearrow\F_\Lambda(\alpha)\cap G\}$.

The following two results give conditions for the topological separability and countability of blur shifts:

\begin{prop}\label{prop:separability}  A blur shift $\Sigma_\Lambda$ is separable if and only if $B_1(\Lambda)$ is countable.

\end{prop}

\begin{proof} 

Let $\mathcal{D}\subset \Sigma_\Lambda$ be a countable dense subset. Then, since $\mathcal{D}$ is dense in $\Sigma_\Lambda$, it follows that for each $a\in B_1(\Lambda)$ there exists $\y\in \mathcal{D}\cap Z_\Lambda(a)$. Since for distinct $a$ we have necessarily a distinct $\y$, then the cardinality of $\mathcal{D}$ is not less than the cardinality of $B_1(\Lambda)$, which implies that $B_1(\Lambda)$ is countable. 

Conversely, if $B_1(\Lambda)$ is countable, then $B(\Lambda)$ is countable. For each $\w=w_1...w_m\in B(\Lambda)$ we pick a point $\y^\w\in Z(\w)$. Thus, $D:=\{\y^\w \in \Sigma_\Lambda:\ \w\in B(\Lambda)\}$ is a countable set. Furthermore, $D$ is dense, since given any $Z_\Lambda(a_0...a_n\bar H,F)$ we can take $\w=a_0...a_nh\in B(\Lambda)$ with $h\in H\setminus F$ and thus $\y^\w\in D\cap Z_\Lambda(a_0...a_n\bar H,F)$.

\end{proof}

\begin{prop}\label{prop:2nd_count}  A  blur shift $\Sigma_\Lambda$ is second countable if and only if $B_1(\Lambda)$  and $\V_\Lambda$ are countable.

\end{prop}

\begin{proof} Suppose $B_1(\Lambda)$ and $\V_\Lambda$ are countable. It follows that  $B(\Lambda)$ is also countable. Hence, the topological basis of $\Sigma_\Lambda$,
$$\mathfrak{B}_{\Sigma_\Lambda}:=\{Z_\Lambda(\w),\ Z_\Lambda(\w\bar H,F):\ \w\in B(\Lambda),\ H\in\V_\Lambda ,\ F\subset \F_\Lambda(\w)\cap H \text{ finite}\}$$ is countable.

The converse, comes directly to the fact that if $B_1(\Lambda)$ or $\V_\Lambda$ were not countable, then $\mathfrak{B}_{\Sigma_\Lambda}$ would have an uncountable number of generalized cylinders of the form $Z_\Lambda(\w)$. Therefore, 
 $\mathfrak{B}_{\Sigma_\Lambda}$ would not be countable neither would have a countable subfamily which is a basis.

\end{proof}

\begin{prop}\label{prop:first_countable} A blur shift $\Sigma_\Lambda$ is first countable if and only if for all $H\in\V_\Lambda$ we have $H\cap B_1(\Lambda)$ countable.
\end{prop}

\begin{proof}

Just note that if there exists some $G\in \V_\Lambda$ such that $G\cap B_1(\Lambda)$ is not countable, then the point $\x=(\tilde G\tilde G\tilde G...)$ has an uncountable local basis $\mathfrak{B}_\x:=\{Z(\bar G,F):\ F\subset G\text{ finite}\}$ which has not any countable subfamily which is a local basis. Thus, $\Sigma_\Lambda$ is not first countable.\\

Conversely, observe that if $\x\in\Lambda$, then $\mathfrak{B}_\x:=\{Z_\Lambda(x_0...x_n):\ n\in\N\}$ is always a countable local basis of $\x$.  On the other hand, given $\x=(x_0...x_{n-1}\h\h\h...)\in\partial\Lambda$, if $\F_\Lambda(x_0...x_{n-1})\cap H$ is countable we can take an enumeration of it, $\F_\Lambda(x_0...x_{n-1})\cap H=\{h_1,h_2,...\}$, and then $\mathfrak{B}_\x:=\{Z_\Lambda(x_0...x_{n-1}\bar H,\{h_1,...,h_\ell\}):\ \ell\ge 0\}$ is a countable local basis of $\x$.

\end{proof}

Define $$B_1(\V_\Lambda):=\{a\in\A:\ a\in B_1(\Lambda)\cap H\text{ for some } H\in\V_\Lambda\}.$$ 

We notice that even if $\Sigma_\Lambda$ is first countable, that is, if $B_1(\Lambda)\cap H$ is countable for all $H\in\V_\Lambda$, it is possible that $\V_\Lambda$ and $B_1(\V_\Lambda)$ are not countable.

\begin{ex} The blur shifts of examples \ref{ex:ultragraph_shift}, \ref{ex:blurshift_count_inf_intersc}, \ref{ex:prime_shift} and \ref{ex:intersect-all_shift} are first and second countable, since they have countable alphabets.

The blur shifts of examples \ref{ex:blurshift_uncount_inf_intersc} and \ref{ex:blurshift_uncount_no_intersc} are both first countable, but not second countable. Observe that while $B_1(\V_\Lambda)$ is countable in Example \ref{ex:blurshift_uncount_inf_intersc}, it is uncountable in Example \ref{ex:blurshift_uncount_no_intersc}.

The blur shift of Example \ref{ex:blurshift_uncount_uncont} is neither first nor second countable.
\end{ex}

\subsection{Metrizability}

To find criteria for the metrizability of blur shifts, we shall use the Nagata-Smirnov metrization theorem, which states that a topological space is metrizable if and only if it is regular, Hausdorff and has a {\bf countably locally finite} basis. We recall that a family $\mathcal{U}$ of sets in a topological space is said to be {\bf locally finite} if any point of the space has an open neighborhood which intersects at most a finite number of sets $\mathcal{U}$, while a family $\mathcal{U}$ of sets is said to be {\bf countably locally finite} if it is the countable union of locally finite families.

Since $\Sigma_\Lambda$ is always regular and Hausdorff (propositions \ref{prop:Hausdorff} and \ref{prop:Regular}), we just need to find conditions under which $\tau_{\Sigma_\Lambda}$ has countably locally finite basis.

The following lemma translates to the language of generalized cylinders the obvious fact that in a Hausdorff and regular space, a family of closed sets is locally finite if and only if any intersection of infinitely many sets of the family is always empty and any convergent sequence is contained at most in a finite union of sets of the family.

\begin{lem}\label{lem:locally_finite} Let $\Sigma_\Lambda$ be a blur shift space. A family of generalized cylinders $\mathfrak{F}\subset\mathfrak{B}_{\Sigma_\Lambda}$ is locally finite if and only if all the following conditions hold:
\begin{enumerate}
\item Any infinite intersection of generalized cylinders of $\mathfrak{F}$ is empty;

\item For all $k\geq 0$, $\alpha\in B_{k}(\Lambda)$ and  $H\in \V_\Lambda$ there exists a finite set $F\subset H$ such that no word $\gamma\in B(\Lambda)$ with $\gamma_0...\gamma_{k-1}=\alpha$ and $\gamma_k\in H\setminus F$ is used in any generalized cylinder of 
 $\mathfrak{F}$.

\item Any infinite subcollection  $\{Z_\Lambda(\gamma \bar G^\ell,F^\ell):\ \ell\in\lambda\}\subset\mathfrak{F}$ is such that for all $H\in\V_\Lambda$ we have $|H\cap\bigcup_{\ell\in\lambda}\big(G^\ell\setminus F^\ell\big)|<\infty$.

\end{enumerate}
\end{lem}

\begin{proof}

Firstly let us check that $i.-iii.$ are necessary conditions to $\mathfrak{F}$ be locally finite. In fact, it is direct that if condition $i.$ does not hold, then we can take an element $\x$ which belongs to the intersection of infinitely many generalized cylinders of $\mathfrak{F}$ and any neighborhood of such $\x$ will intersect all of those infinitely many generalized cylinders. 
Now, suppose there are $\alpha\in B_{k}(\Lambda)$ and  $H\in \V_\Lambda$ and suppose an infinite subset $\{h^i\}_{i\geq 1}\subset H$ with the property that for each $i\geq 1$ there exists $\gamma^i\in B(\Lambda)$ which is used in the definition of some generalized cylinder $Z^i\in\mathfrak{F}$, and which is such that $\gamma_0^i...\gamma_k^i=\alpha h^i$.
Therefore, any neighborhood $Z_\Lambda(\alpha \bar H,F)$ of the point $(\alpha_0...\alpha_{k-1}\h\h\h...)$ will intersect infinitely many generalized cylinders $Z^i$. 
 Finally, suppose there exist $\{Z_\Lambda(\gamma \bar G^\ell,F^\ell):\ \ell\in\lambda\}\subset\mathfrak{F}$ and $H\in\V_\Lambda$ such that we have $|H\cap\bigcup_{\ell\in\lambda}\big(G^\ell\setminus F^\ell\big)|=\infty$. It implies that $|H \cap \F_\Lambda(\gamma)|=\infty$ and for all finite $F\subset H \cap \F_\Lambda(\gamma)$ we have $Z_\Lambda(\gamma\bar H, F)$ intersecting infinitely many generalized cylinders of $\{Z_\Lambda(\gamma \bar G^\ell,F^\ell):\ \ell\in\lambda\}$.\\

To check that $i.-iii.$ are sufficient conditions for $\mathfrak{F}$ be locally finite, note that for any $\x\in\Lambda$, if for all $n\geq 0$ the generalized cylinder $Z_\Lambda(x_0...x_n)$ intersects infinitely many generalized cylinders of $\mathfrak{F}$ means that $\x$ should  belong to infinitely many generalized cylinders of $\mathfrak{F}$, contradicting condition $i.$. For $\x\in\partial\Lambda$, say $\x=(\alpha_0...\alpha_{k-1}\h\h\h...)$,  if any neighborhood $Z_\Lambda(\alpha_0...\alpha_{k-1}\bar H, F)$ intersects infinitely many generalized cylinders of $\mathfrak{F}$, then at least one of the three situations holds: $(a)$ $\x$ belongs to infinitely many generalized cylinders of $\mathfrak{F}$; $(b)$ For any finite $F\subset H$,  $Z_\Lambda(\alpha_0...\alpha_{k-1}\bar H, F)$  intersects infinitely many generalized cylinders of $\mathfrak{F}$, each one of them fixing the $k$ first entries as $x_0...x_{k-1}$ and with the $(k+1)^{th}$ entry belonging to $H$; $(c)$  $Z_\Lambda(\alpha_0...\alpha_{k-1}\bar H, F)$ intersects infinitely many cylinders $Z_\Lambda(\alpha_0...\alpha_{k-1}\bar G^\ell, F^\ell)$. However, situation $(a)$ is avoided by condition $i.$, situation $(b)$ is avoided by condition $ii.$, and situation $(c)$ cannot occur due to condition $iii.$.

\end{proof}

\begin{theo}\label{theo:Metrizable} Suppose $\Sigma_\Lambda$ is a blur shift which is first countable and such that at least one of the following conditions holds:
\begin{enumerate}

\item\label{theo:Metrizable_i} $\V_\Lambda$ is countable;

\item\label{theo:Metrizable_ii} Each $H\in \V_\Lambda$ has just a finite number of elements that appear in some other set of $\V_\Lambda$ (but it is possible that some element appears in infinitely many sets of $\V_\Lambda$). 

\end{enumerate}

Then $\Sigma_\Lambda$ is metrizable.

\end{theo}

\begin{proof}

We just need to prove that if $\Sigma_\Lambda$ is first countable  (that is, for all $H\in \V_\Lambda$ we have $H\cap B_1(\Lambda)$ countable - Proposition \ref{prop:first_countable}) and $i.$ or $ii.$ holds, then $\tau_{\Sigma_\Lambda}$ has a countably locally finite basis.\\

First, we observe that if $\V_\Lambda$ is empty, then $\Sigma_\Lambda=\Lambda$ which is metrizable (in this case both $i.$ and $ii.$ hold trivially). Hence, we will assume $\V_\Lambda$ nonempty.\\

Recall that a basis of $\tau_{\Sigma_\Lambda}$ is $\mathfrak{B}_{\Sigma_\Lambda}:=\mathfrak{C}\cup\mathfrak{G}$, where 
 $$\mathfrak{C}:=\{Z_\Lambda(\alpha): \alpha \in B(\Lambda)\}$$
 and
 $$\mathfrak{G}:=\{Z_\Lambda(\alpha \bar H ,F): \alpha \in B(\Lambda),\ H\in\V_\Lambda,\ F \subset \F_\Lambda(\alpha)\cap H \text{ finite }\}.$$

\begin{enumerate}
\item Since $\Sigma_\Lambda$ is first countable, we have that $\V_\Lambda$ countable implies $B_1(\V_\Lambda)$ is countable.\\

Consider an enumeration $B_1(\V_\Lambda)=\{a^1,a^2,a^3,...\}$ and for each $k\ge 1$ define $P_k:=\{a^k,a^{k+1},a^{k+2},...\}$. Now, for each $k,n\geq 1$ define
$$\mathfrak{C}_{k,n}:=\{Z_\Lambda(\alpha):\ \alpha\in B_n(\Lambda)\text{ contains just symbols of } B_1(\Lambda)\setminus P_k\}.$$
It follows that each $\mathfrak{C}_{k,n}$ satisfies all the three conditions of Lemma \ref{lem:locally_finite} (condition $iii.$ is satisfied by vacuity).

Hence,  $$\mathfrak{C}=\bigcup_{k,n\geq 1}\mathfrak{C}_{k,n}$$ is countably locally finite.

Now, for each $H\in\V_\Lambda$ consider an enumeration of $B_1(\Lambda)\cap H=\{h^1,h^2,h^3,...\}$ and define $F_\ell^H:=\{h^1,..,h^\ell\}$. For each $k\geq 1$, $\ell,n\geq 0$ and $H\in\V_\Lambda$, define
 $$\mathfrak{G}_{k,\ell,n,H}:=\{Z_\Lambda(\alpha\bar H,F_\ell^H):\ \alpha\in B_n(\Lambda)\text{ contains just symbols of } B_1(\Lambda)\setminus P_k,\ |\F_\Lambda(\alpha)\cap H|=\infty\}.$$
We have that $\mathfrak{G}_{k,\ell,n,H}$ satisfies all the three conditions of Lemma \ref{lem:locally_finite} and so it is locally finite. Thus, 
$$\mathfrak{G}=\bigcup_{k\geq 1,\ \ell,n\geq 0,\ H\in\V_\Lambda}\mathfrak{G}_{k,\ell,n,H}$$  is countably locally finite.\\

Therefore, $\mathfrak{B}_{\Sigma_\Lambda}$ is a  countably locally finite basis, and so $\Sigma_\Lambda$ is metrizable.

\item If $\V_\Lambda$ is countable, then the result follows from the case $i.$. Suppose $\V_\Lambda$ is an uncountable family. Given $H\in\V_\Lambda$, define
 \begin{equation}\label{eq:S_H equiv}S_H:=\{h\in H:\ \exists G\in\V_\Lambda\text{ s.t. } G\neq H\text{ and } h\in G\}=H\cap\bigcup_{G\in\V_\Lambda,\ G\neq H} G,\end{equation}
which is finite by hypothesis.

Since $\Sigma_\Lambda$ is first countable, for each $H\in \V_\Lambda$ we can take an enumeration of $B_1(\Lambda)\cap H:=\{h^1,h^2,h^3,...\}$ and for $n\geq 1$ we define $F_n^H:=\{h^1,h^2,...,h^n\}$ and $O_n^H:=\{h^n,h^{n+1},h^{n+2},...\}$. Now, for each $n\geq 1$, consider the set 
$$Q_n:=\bigcup_{\ H\in\V_\Lambda}O_n^H.$$
We remark that $Q_n$ is an uncountable set, since there are uncountable many sets in $\V_\Lambda$.

For each $k,n\geq 1$ define
$$\mathfrak{C}_{k,n}:=\{Z_\Lambda(\alpha):\ \alpha\in B_n(\Lambda)\text{ contains just symbols of } B_1(\Lambda)\setminus Q_k\},$$
which is a locally finite family. In fact, conditions $i.$ and $iii.$ of Lemma \ref{lem:locally_finite} follow directly, while  condition $ii.$ follows from the fact that any word $\alpha$ defining a generalized cylinder $Z_\Lambda(\alpha)$ is such that it can use only a finite number of symbols lying in any $H\in \V_\Lambda$. Indeed, with respect to a given $H$, $\alpha$ can use at most the $k-1$ symbols of $F_{k-1}^H$, and the finite number of elements in $\bigcup_{G\in\V_\Lambda}\big(S_H\cap F_{k-1}^G\big)$.

Thus,  $$\mathfrak{C}=\bigcup_{k,n\geq 1}\mathfrak{C}_{k,n}$$ is a countably locally finite family.\\

Now, for each $k\geq 1$ and $\ell,n\geq 0$ define
 $$\mathfrak{G}_{k,\ell,n}:=\{Z_\Lambda(\alpha\bar H,F_\ell^H\cup S_H):\ H\in\V_\Lambda,\ \alpha\in B_n(\Lambda)\text{ has just symbols of } B_1(\Lambda)\setminus Q_k,\ |\F_\Lambda(\alpha)\cap H|=\infty\}.$$
It follows that though $\mathfrak{G}_{k,\ell,n}$ contains uncountable many sets, it is a locally finite family. In fact, condition $i.$ of  Lemma \ref{lem:locally_finite} follows from the fact that we can only have subfamilies of generalized cylinders $\{Z_\Lambda(\alpha^t\bar H^t,F_\ell^{H^t}\cup S_{H^t}):\ t\geq 1\}\subset \mathfrak{G}_{k,\ell,n}$ if and only if $\alpha^t=\alpha^s$ for all $s,t\geq 1$ and $\bigcap_{t\geq 1}  \big(H^t\setminus (F_\ell^{H^t}\cup S_{H^t})\big) \neq\emptyset$, which is avoided by the fact that for fixed $\alpha$, each $H^t$ of $\V_\Lambda$ is used just once in a generalized cylinder of $\mathfrak{G}_{k,\ell,n}$, and $H^t\setminus (F_\ell^{H^t}\cup S_{H^t})$ is disjoint of $H^s\setminus (F_\ell^{H^s}\cup S_{H^s})$ if $t\neq s$. Condition $ii.$ of Lemma \ref{lem:locally_finite} is checked as made for $\mathfrak{C}_{k,n}$. Finally, condition $iii.$ of Lemma \ref{lem:locally_finite} comes from \eqref{eq:S_H equiv}.

Hence, we can define  $$\hat{\mathfrak{G}}=\bigcup_{k\geq 1,\ \ell,n\geq 0}\mathfrak{G}_{k,\ell,n}$$ which is a countably locally finite family, and so $$\hat{\mathfrak{B}}_{\Sigma_\Lambda}:=\mathfrak{C}\cup\hat{\mathfrak{G}}$$ is also a countably locally finite family.

We notice that $\hat{\mathfrak{G}}$ is also a basis for the topology $\tau_{\Sigma_\Lambda}$ in spite of $\hat{\mathfrak{G}}\subsetneq\mathfrak{G}$. In order to check this, it is sufficient to observe that any $Z_\Lambda(\alpha_0...\alpha_{k-1} \bar H, F)\in \mathfrak{G}$ can be written as
$$Z_\Lambda(\alpha_0...\alpha_{k-1} \bar H, F)=Z_\Lambda(\alpha_0...\alpha_{k-1} \bar H, F_\ell^H\cup S_H)\cup\bigcup_{h\in (F_\ell^H\cup S_H)\setminus F} Z_\Lambda (\alpha_0...\alpha_{k-1} h)$$ where $F_\ell^H$ is chosen such that $F\subset F_\ell^H\cup S_H$.

\end{enumerate}

\end{proof}

The next corollary can be obtained as a direct consequence of Proposition \ref{prop:2nd_count} and Theorem \ref{theo:Metrizable}.\ref{theo:Metrizable_i} (as well as it could be derived from propositions \ref{prop:Hausdorff} and \ref{prop:Regular}, and then applying the Urysohn metrization theorem):

\begin{cor}\label{cor:2nd_count-metr} If a blur shift is second countable, then it is metrizable.
\end{cor}

\qed

\begin{ex}\label{ex:metrizable_blur_shifts} The full blur shifts of examples \ref{ex:blurshift_count_inf_intersc}, \ref{ex:prime_shift}, \ref{ex:intersect-all_shift} and \ref{ex:ultragraph_shift} are metrizable due to Corollary \ref{cor:2nd_count-metr}. 
The blur shift of Example \ref{ex:blurshift_uncount_inf_intersc} is metrizable due to Theorem \ref{theo:Metrizable}.\ref{theo:Metrizable_i}, while the blur shift of Example \ref{ex:blurshift_uncount_no_intersc} is metrizable due to Theorem \ref{theo:Metrizable}.\ref{theo:Metrizable_ii}.
\end{ex}

We remark that contrarily to what happens in the particular contexts of Ott-Tomforde-Willis shifts (where metrizability is equivalent to countability), in the general context of blur shifts it remains open to determine a complete characterization of the metrizable shifts. We remark that Ott-Tomforde-Willis shifts are blur shifts constrained to have at most one blurred set which is all the alphabet. In the general context, when we can consider any quantity of blurred sets and cardinality for the alphabet, Theorem \ref{theo:Metrizable} shows that such characterization is a bit more complicated since it depends on finding a general strategy under which one can characterize whether or not a topology has a countably locally finite basis. In Subsection \ref{subsec:metric} we will construct a metric for blur shifts over countable alphabets, however it also remains an open problem to present an explicit metric for other cases of metrizable blur shifts.

Although it remains an open problem to give a complete characterization of metrizable blur shifts, for the particular cases of compact blur shifts we can obtain such  characterization of metrizability. The following corollary uses two results that will be proved in Subsection \ref{subsec:Blur_shift_compactfication} to prove a result analogous to  \cite[Corollary 2.18]{Ott_et_Al2014}. In particular, it shows a relationship between first countability, metrizability and compactness.

\begin{cor}\label{cor:equiv_count_metr}
Let $\Sigma_\Lambda$ be a compact blur shift. The following statements are equivalent:
\begin{enumerate}
\item $B_1(\Lambda)$ is countable;

\item $\Sigma_\Lambda$ is first countable;

\item $\Sigma_\Lambda$ is second countable;

\item $\Sigma_\Lambda$ is separable;

\item $\Sigma_\Lambda$ is metrizable.
\end{enumerate}
\end{cor}

\begin{proof} \phantom\\

\begin{description}

\item[$iii.\Rightarrow ii.$] It is direct.

%\item[$i.\Rightarrow ii.$] From Proposition \ref{prop:first_countable}, if $B_1(\Lambda)$ is countable, then $\Sigma_\Lambda$ is first countable.

\item[$ii.\Rightarrow i.$]  Since $\Sigma_\Lambda$  is first countable, then from Proposition \ref{prop:first_countable} we have that  $H\cap B_1(\Lambda)$ is countable for all $H\in\V_\Lambda$. 
On the other hand,
from the compactness of $\Sigma_\Lambda$, applying Theorem \ref{theo:compactness}, we get that $\V_\Lambda$ is a finite family that covers $B_1(\Lambda)\setminus  \{a_1,...,a_\ell\}$, for some a finite set $\{a_1,...,a_\ell\}\subset B_1(\Lambda)$. Therefore $B_1(\Lambda)=\{a_1,...,a_\ell\}\cup\bigcup_{H\in\V_\Lambda}(H\cap B_1\big(\Lambda)\big)$ is countable.

\item[$i.\Rightarrow iii.$] From Theorem \ref{theo:compactness} the compactness of $\Sigma_\Lambda$ implies $\V_\Lambda$ being finite. Thus, from Proposition \ref{prop:2nd_count}, if $B_1(\Lambda)$ is also countable, then $\Sigma_\Lambda$ is second countable.

\item[$iii.\Rightarrow v.$] From Corollary \ref{cor:2nd_count-metr}.

\item[$v.\Rightarrow i.$] By hypothesis $\Sigma_\Lambda$ is compact. If it is also metrizable, then from Corollary \ref{cor:compact_matrizable} we get that the alphabet shall be countable.

\item[$i.\Leftrightarrow iv.$] It is given by Proposition \ref{prop:separability}.

\end{description}
\end{proof}

 \subsection{Second-countable blur shifts and metrics}\label{subsec:metric}
 
In this subsection, we will assume $\Sigma_\A$ being second countable and we will construct a family of metrics for the topology $\tau_{\Sigma_{\A^\N}}$ (and so for any blur shift over the same alphabet and same blurred sets). Such construction follows ideas of \cite{GOUG2020-2} and \cite{Ott_et_Al2014}, but uses more direct and intuitive construction than the ones presented in these references. 

We recall that the case $\Sigma_\A$ being second countable corresponds to the particular case of metrizable blur shifts given by Corollary \ref{cor:2nd_count-metr}. It remains an open problem to construct metrics for other cases.\\
 
Let \begin{equation}\label{eq:prefix}\mathfrak{p}:=\{(\alpha), (\alpha\bar H):\ \alpha\in B(\A^\N),\  \bar H\in\VV\}.\end{equation}
 
\begin{defn}\label{defn:prefix_sequences} We say that $p\in\mathfrak{p}$ is a {\bf prefix} of a sequence $\x$ if and only if $\x\in Z(p)$.
Given $p,q\in\mathfrak{p}$ we will say that $p$ is prefix of $q$ if and only if $Z(q)\subset Z(p)$.

\end{defn}
 
Note that, using the language of prefixes, we have that the cylinders $Z(\alpha)$ and $Z(\alpha \bar H)$ are the sets of all sequences in $\Sigma_{\A^\N}$ that have $(\alpha)$ and $(\alpha \bar H)$ as prefix, respectively. Furthermore, any element of $\A^\N$ has infinitely many prefixes in $\mathfrak{p}$, while an element of $\mathcal{L}_n(\Lambda)$, with $n\in\N$, has exactly $n+1$ prefixes in $\mathfrak{p}$.

From Proposition \ref{prop:2nd_count} both $\A$ and $\V$ are countable, and then $\mathfrak{p}$ is also countable. Consider an enumeration of $\mathfrak{p}$: $$\mathfrak{p}:=\{p_1,p_2,p_3,\ldots\}.$$
For each specific enumeration of $\mathfrak{p}$, we define $d:\Sigma_{\A^\N}\times\Sigma_{\A^\N}\to[0,+\infty)$ by
\begin{equation}\label{eq:defn_metric}d(\x,\y) := \left\{\begin{array}{lcl} 1/2^i &,&  \text{where $i \in \N$ is the smallest integer such that $p_i$ is prefix either of  $\x$ or of $\y$}, \\ 0 &,& \text{if $\x = \y$.}\end{array}\right.\end{equation}

Let us check that $d$ is a metric on $\Sigma_{\A^\N}$. It is direct that $d(\x,\y)=d(\y,\x)\ge 0$ for all $\x,\y\in \Sigma_{\A^\N}$, and $d(\x,\y)= 0$ if and only if $\x=\y$. To check the triangular inequality, given $\x,\y,\z\in\Sigma_{\A^\N}$, suppose that $d(\x, \z) = 2^{-i},$ where $i$ is the first index such that $p_i$ is a prefix of only one of $\x$ or $\z$. Without loss of generality, assume that $p_i$ is a prefix of $\x$ and is not of $\z$. We have then two possibilities: Either $p_i$ is a prefix of $\y$ or it is not. In the former case, we have $d(\y, \z)=1/2^j$ where $j=\min\{k\geq 1:\ p_k \text{ is prefix either of  } \y \text{ or of } \z\}\leq i$, and therefore $d(\x, \z)\leq d(\y, \z)$. In the latter case, when $p_i$ is not a prefix of $\y$, then $d(\x, \y)=1/2^\ell$, where $\ell=\min\{k\geq 1:\ p_k \text{ is prefix either of  } \x \text{ or of } \y\}\leq i$, and so we have that $d(\x, \z) \leq d(\x, \y)$. Hence, we conclude that $d(\x, \z)\leq d(\x, \y)+d(\y, \z)$.\\

Denote as $B_{d}(\x,\varepsilon)$ the open ball for the metric $d$, centered in $\x$ and with radius $\varepsilon>0$. Given $\varepsilon>0$, let $k$ be the smallest positive integer such that $1/2^k< \varepsilon$. Observe that a point $\y\neq\x$ belongs to $B_{d}(\x,\varepsilon)$ if and only if the first $p_j\in\mathfrak{p}$ which is prefix either of $\x$ or of $\y$ is such that $j\geq k$. \\

\begin{prop} The metric $d$, defined from any enumeration of $\mathfrak{p}$, and the generalized cylinders generate the same topology.
\end{prop}

\begin{proof}

First we need to show that for any given $\x\in\Sigma_{\A^\N}$ and a generalized cylinder of its local basis $\mathfrak{B}_{\x}$, there exists $\varepsilon>0$ such that  $B_{d}(\x,\varepsilon)$ is contained in that generalized cylinder. In fact, if $\x\in\A^\N$, and the generalized cylinder is $Z(x_0...x_n)$, then we just need take $k\geq 1$ such that $p_k\in\mathfrak{p}$ is the prefix $p_k=(x_0...x_n)$. Therefore, $\y\in B_{d}(\x,1/2^k)$ implies that the first $p_j\in\mathfrak{p}$ which is prefix either of $\x$ or of $\y$ is such that $j > k$, and so $p_k$ is also prefix of $\y$, which implies that $B_{d}(\x,1/2^k)\subset Z(p_k)=Z(x_0...x_n)$. 
If $\x\in\partial\A^\N$ and the generalized cylinder is $Z(x_0...x_{n-1}\bar H,F)$, take $k\geq 1$ as the greatest index such that  $p_k\in \{p\in\mathfrak{p}:\ p= (x_0...x_{n-1}h)\text{ with }h\in F\cup\{\bar H\}\}$. Hence, 
if $\y\in B_{d}(\x,1/2^k)$, since the first $p_j\in\mathfrak{p}$ which is prefix either of $\x$ or of $\y$ is such that $j > k$, and since $p_k$ comes later than $(x_0...x_{n-1}\bar H)$, prefix of $\x$, we conclude that $(x_0...x_{n-1}\bar H)$ shall also be prefix of $\y$. Furthermore, $p_k$ comes after all prefixes $(x_0...x_{n-1}h)$ with $h\in F$ which are not prefixes of $\x$ and therefore they could not be prefixes of $\y$ either. These facts imply that $\y\in Z(x_0...x_{n-1}\bar H,F)$, and so  $B_{d}(\x,1/2^k)\subset Z(x_0...x_{n-1}\bar H,F)$.\\

Now, we will show that for any given $\x\in\Sigma_{\A^\N}$ and $\varepsilon>0$, there exists a generalized cylinder in $\mathfrak{B}_{\Sigma_{\A^\N}}$ which contains $\x$ and is contained in $B_{d}(\x,\varepsilon)$. Indeed, if $\x\in\A^\N$, we can take $k\geq 1$ such that $1/2^k\leq\varepsilon$, $p_k$ is prefix of $\x$, and $p_k$ is not prefix of any $p_j$ with $j<k$ (we recall it is always possible to find such $k$ since elements of $\A^\N$ have infinitely many prefix in $\mathfrak{p}$). Thus, any $\y\in Z(p_k)$ is such that $p_k$ is its prefix and for each $j<k$ a prefix $p_j$ either is prefix of both $\x$ and $\y$ or is not prefix of any of them. Hence, we have that the first $j$ such $p_j$ is either prefix of $\x$ or of $\y$ is necessarily greater than $k$, and so $\y\in B_{d}(\x,1/2^k)\subset B_{d}(\x,\varepsilon)$, which means $Z(p_k)\subset B_{d}(\x,\varepsilon)$. On the other hand, if $\x=(x_0...x_{n-1}\h\h\h...)\in\partial\A^\N$, then we just need to take $k\geq 1$ such that $1/2^k\leq\varepsilon$ and consider the generalized cylinder $Z(x_0...x_{n-1}\bar H,F)$ where $F:=\{h\in H:\ \exists j\leq k\text{ s.t. }p_j=(x_0...x_{n-1}h...x_{n_j})\text{ or } p_j=(x_0...x_{n-1}\bar G)\text{ with }G\neq H\text{ and }h\in G\}$. Therefore, if $\y\in Z(x_0...x_{n-1}\bar H,F)$ it follows that any prefix of $\x$ is also a prefix of $\y$, and there is not a $j\leq k$ such that $p_j$ is prefix of $\y$ but not of $\x$. Thus, the first index $j$ such $p_j$ is prefix of $\y$ but no of $\x$ is greater than  $k$, which implies $\y\in B(\x,1/2^k)\subset  B_{d}(\x,\varepsilon)$ which leads to conclude that 
$Z(x_0...x_{n-1}\bar H,F)\subset  B_{d}(\x,\varepsilon)$.

\end{proof}

\subsection{Compactness and local-compactness criteria}\label{subsec:Blur_shift_compactfication}
  
In this subsection we shall present necessary and sufficient conditions under which a blur shift $\Sigma_\Lambda$ is (sequentially) compact or locally compact.\\

Firstly, we shall prove that sequential compactness and compactness are equivalent concepts in the context of blur shifts (in spite of them being metrizable or not).

Recall that the induced topology on some given $\mathfrak{X}\subset\Sigma_{\A^\N}$ has a basis $\mathfrak{B}_\mathfrak{X}$ composed by the sets $Z_\mathfrak{X}(\alpha):=Z(\alpha)\cap \mathfrak{X}$ and $Z_\mathfrak{X}(\alpha\bar H, F):=Z(\alpha\bar H,F)\cap \mathfrak{X}$ for $\alpha\h\in B(\mathfrak{X})$, and a finite set $F\subset \F_\mathfrak{X}(\alpha)\cap H$. Furthermore, $B_n(\mathfrak{X})$ stands for the set of all words with length $n$ in $\mathfrak{X}$, $B(\mathfrak{X})$ stands for the set of all finite words in $\mathfrak{X}$, and $\F_{\mathfrak{X}}(a_0\ldots a_{n-1})$ stands for the follower set of $a_0\ldots a_{n-1}\in B(\mathfrak{X})$ in $\mathfrak{X}$.

\begin{rmk}\label{rmk:finite_subcover} For all $\alpha\in B(\A^\N)$, $H\in \V$, and any family of finite sets $\{F_\ell\}_ {\ell\in\lambda}$, there exist $F_1,...,F_t\in \{F_\ell\}_ {\ell\in\lambda}$ such that $$\bigcup_{\ell\in\lambda}Z(\alpha\bar H,F_\ell)=\bigcup_{i=1}^tZ(\alpha\bar H,F_i).$$
Furthermore, taking $F:=\bigcap_{i=1}^tF_i$, we have  $$\bigcup_{\ell\in\lambda}Z(\alpha\bar H,F_\ell)=Z(\alpha\bar H,F).$$
\end{rmk}

In order to characterize compactness and local compactness of blur shift spaces we shall first  prove that sequential compactness and compactness are equivalent concepts in our topology. To achieve this we need the following auxiliary lemma.

\begin{lem}\label{lem:seq_compc<=>finite_follower} Let $\mathfrak{X}\subset\Sigma_{\A^\N}$ be any subset (not necessarily a blur shift space). If  $\mathfrak{X}$ is sequentially compact, then for all $\alpha\in B(\mathfrak{X})\cap B(\A^\N)$ the sets $V_\alpha:=\{H\in\V:\ \h\in\F_\mathfrak{X}(\alpha)\}$ and  $W_\alpha:=\F_\mathfrak{X}(\alpha)\setminus \left( \bigcup_{H\in V_\alpha}\bar H\right)$ are finite. Conversely, under the additional assumption that $\mathfrak{X}$ is closed, if for all $\alpha\in B(\mathfrak{X})\cap B(\A^\N)$ the sets $V_\alpha$ and  $W_\alpha$ are finite, then $\mathfrak{X}$ is sequentially compact.
\end{lem}

\begin{proof}

Suppose $\mathfrak{X}$ sequentially compact and let us show that for each $a_1...a_{n-1}\in B(\mathfrak{X})\cap B(\A^\N)$ there are at most finitely many  $\h\in\V$ such that $(a_1...a_{n-1}\h\h\h...) \in \mathfrak{X}$. Indeed, if we suppose, by contradiction, that for some $a_1...a_{n-1}\in B(\mathfrak{X})$ there exist an infinite set $\{\h^\ell\in\V:\ell\geq 0\}$ such that $(a_1...a_{n-1}\h^\ell\h^\ell\h^\ell...) \in \mathfrak{X}$ for all $\ell\geq 0$, then $(a_1...a_{n-1}\h^\ell\h^\ell\h^\ell...)_{\ell\geq 0}$ is a sequence in $\mathfrak{X}$ which has not any convergent subsequence, contradicting that $\mathfrak{X}$ is sequentially compact. Hence, given $\alpha\in B(\mathfrak{X})$, the set $V_\alpha:=\{H\in\V:\ 
\h\in \F_\mathfrak{X}(\alpha)\}$ is finite (possibly empty).

Now, given $\alpha=a_0...a_{n-1}\in B(\mathfrak{X})\cap B(\A^\N)$, define $W_\alpha:=\F_\mathfrak{X}(\alpha)\setminus  \left(\bigcup_{H\in V_\alpha}\bar H\right)$. Note that $W_\alpha$ just contains symbols of the alphabet $\A$. Suppose by contradiction  $W_\alpha$ is infinite, and for each $\ell\geq 0$ take $w^\ell\in W_\alpha\setminus\{w^j:\ j<\ell\}$. Hence, there exists a sequence $(\x^\ell)_{\ell\geq 0}\in\mathfrak{X}$ such that $x_0^\ell...x_{n-1}^\ell=\alpha$ and $x_n^\ell=\w^\ell$. Since for each distinct $\ell$, we have a distinct $x^\ell_n$ which does not belong to any blurred set, it follows that  $(\x^\ell)_{\ell\geq 0}$ does not contain any convergent subsequence, which contradicts the sequential compactness of $\mathfrak{X}$. Thus, $W_\alpha$ shall be finite.\\

Now suppose $\mathfrak{X}$ closed and that for all $\alpha\in B(\mathfrak{X})\cap B(\A^\N)$ the sets $V_\alpha$ and $W_\alpha$ are finite. Let $(\x^\ell)_{\ell\geq 0}\in \mathfrak{X}$ be any sequence. If $(\x^\ell)_{\ell\geq 0}$ has a subsequence $(\x^{\ell_{0_k}})_{k\geq 0}$ such that $x^{\ell_{0_j}}_0=x^{\ell_{0_k}}_0$ for all $j,k\geq 0$, then we take this subsequence. Now, if $(\x^{\ell_{0_k}})_{k\geq 0}$ has a subsequence $(\x^{\ell_{1_k}})_{k\geq 0}$ such that $x^{\ell_{1_j}}_1=x^{\ell_{1_k}}_1$ for all $j,k\geq 0$, then we take this subsequence. We proceed recursively, either infinitely or until not to be able of finding a subsequence $(\x^{\ell_{n_k}})_{k\geq 0}$ of $(\x^{\ell_{{n-1}_k}})_{k\geq 0}$ such that $x^{\ell_{n_j}}_n=x^{\ell_{n_k}}_n$ for all $j,k\geq 0$. If there always be such subsequences, then we will have an infinite family of subsequences \begin{equation}\label{eq:subseq}\big\{(\x^{\ell_{n_k}})_{k\geq 0}\big\}_{n\geq 0},\end{equation} and it follows that $(\x^{\ell_{n_1}})_{n\geq 0}$ is a subsequence of  $(\x^\ell)_{\ell\geq 0}$ which converges in $\mathfrak{X}$ (since  $\mathfrak{X}$ is closed).
 
On the other hand, suppose that it is not possible to construct an infinite family as \eqref{eq:subseq}, and let $n\geq 0$ be the first integer for which $(\x^{\ell_{{n-1}_k}})_{k\geq 0}$ is such that $x^{\ell_{{n-1}_j}}_n\neq x^{\ell_{{n-1}_k}}_n$ for all $j,k\geq N$ for some $N\in\N$ (recall that from its construction we have that $x^{\ell_{{n-1}_j}}_i= x^{\ell_{{n-1}_k}}_i=:b_i$ for all $j,k\geq 0$ and for all $0\leq i\leq n-1$). Denote $\beta:=b_1...b_{n-1}$. Since $W_\beta$ finite, then for all but finitely many indexes $j\geq 0$ we have  $x^{\ell_{{n-1}_j}}_n\notin W_\beta$. On the other hand, since $V_\beta$ is a finite family of sets, then infinitely many $x^{\ell_{{n-1}_j}}_n$ belong to the same $H\in V_\beta$, that is, there exists a subsequence 
$(\x^{\ell_{{n}_k}})_{k\geq 0}$ of $(\x^{\ell_{{n-1}_k}})_{k\geq 0}$ such that $x^{\ell_{{n}_k}}_n\in H\setminus \{x^{\ell_{{n}_0}}_n,...,x^{\ell_{{n}_{k-1}}}_n\}$ for all $k\geq 0$. Thus, 
$x^{\ell_{{n}_k}}_n\to\h$ as $k\to \infty$, which implies that $\x^{\ell_{{n}_k}}\to (\beta \h\h\h...)\in\mathfrak{X}$ (observe that we did not need to use here the hypothesis of $\mathfrak{X}$ being closed, since from the definition of $V_\beta$ we have that $(\beta \h\h\h...)\in\mathfrak{X}$).

\end{proof}

Now we can prove the equivalence between compactness and sequential compactness. We remark that as in the previous lemma, the following theorem does not require the set $\mathfrak{X}$ being a blur shift space.

\begin{theo}\label{theo:compact_seq-compact} 
A subset $\mathfrak{X}\subset \Sigma_{\A^\N}$ is compact if and only if it is sequentially compact.
\end{theo}

\begin{proof} 

Suppose that $\mathfrak{X}$ is compact and let $(\x^\ell)_{\ell\geq 0}\in \mathfrak{X}$ be a sequence. As done in the second part of Lemma \ref{lem:seq_compc<=>finite_follower}, if $(\x^\ell)_{\ell\geq 0}$ has a subsequence $(\x^{\ell_{0_k}})_{k\geq 0}$ such that $x^{\ell_{0_j}}_0=x^{\ell_{0_k}}_0$ for all $j,k\geq 0$, we take this subsequence. We proceed recursively, either infinitely or until not to be able of finding a subsequence $(\x^{\ell_{n_k}})_{k\geq 0}$ of $(\x^{\ell_{{n-1}_k}})_{k\geq 0}$ such that $x^{\ell_{n_j}}_n=x^{\ell_{n_k}}_n$ for all $j,k\geq 0$. As before, if we can proceed infinitely, then we get an infinite family of subsequences \eqref{eq:subseq} and we can define a convergent subsequece. On the other hand, if we cannot proceed infinitely, and  $n\geq 0$ is the first integer for which $(\x^{\ell_{{n-1}_k}})_{k\geq 0}$ is such that $x^{\ell_{{n-1}_j}}_n\neq x^{\ell_{{n-1}_k}}_n$ for all $j,k\geq N$ for some $N\in\N$, we consider the open cover of $\mathfrak{X}$,
$$C_n:=\{Z_\mathfrak{X}(a_0...a_n),\ Z_\mathfrak{X}(b_0...b_{n-1}\bar H):\ a_0...a_n,\ b_0...b_{n-1}\in B(\mathfrak{X})\cap B(\A^\N),\text{ and } \h\in\F_\mathfrak{X}(b_0...b_{n-1})\}.$$ 
From the compactness of $\mathfrak{X}$ there is $C_n'\subset C_n$ a finite subcover of  $\mathfrak{X}$, and then there is a subsequence of $(\x^{\ell_{{n-1}_k}})_{k\geq 0}$ which lies in some cylinder $Z_\mathfrak{X}(b_0...b_{n-1}\bar H)\in C_n'$. 
Hence, since $x^{\ell_{{n-1}_k}}_n$ is different for each $k\geq N$, it implies that such subsequence converges to $(b_0...b_{n-1}\h\h\h...)\in \mathfrak{X}$.\\

Now suppose that $\mathfrak{X}$ is sequentially compact. Firstly, note that if $\mathcal{C}$ is a disjoint open cover, then $\mathcal{C}$ shall be finite. In fact, if by contradiction we suppose that $\mathcal{C}$ is infinite, then we can take a sequence $(\x^\ell)_{\ell\geq 0}$ such that each $\x^\ell$ belongs to a different set of $\mathcal{C}$. Therefore, since  $\mathfrak{X}$ is sequentially compact, there exists a subsequence $(\x^{\ell_k})_{k\geq 0}$ which converges to some point $\bar\x\in\mathfrak{X}$, which implies that the set $A\in\mathcal{C}$ which contains $\bar\x$ also contains infinitely many terms of $(\x^{\ell_k})_{k\geq 0}$, contradicting that $\mathcal{C}$ is a disjoint family and each $\x^\ell$ belongs to a distinct set of $\mathcal{C}$.

For the general case, let $\mathcal{C}$ be an open cover, and we can assume, without loss of generality, that its elements are generalized cylinders of $\mathfrak{B}_\mathfrak{X}$. Consider the subcover $\mathcal{C}'$ formed by all maximal sets of $\mathcal{C}$, that is, 
$$\mathcal{C}':=\{Z\in\mathcal{C}:\ \nexists Y\in\mathcal{C}\text{ s.t. } Z\subset Y\}.$$

Let us prove that $\mathcal{C}'$ is finite. To achieve this, we shall construct an open cover $\mathcal{\tilde C}$ which is finite if and only if $\mathcal{C}'$ is finite.\\

In what follows we shall denote as $\alpha$ a word in $\B(\mathfrak{X})\cap B(\A^\N)$. Hence, firstly note that a generalized cylinder of type $Z_\mathfrak{X}(\alpha)$ in $\mathcal{C}'$ is disjoint of any other cylinder in $\mathcal{C}'$. Thus, we have that a generalized cylinder of type $Z_\mathfrak{X}(\alpha\bar H, F_H)$ of $\mathcal{C}'$ can only intersect a generalized cylinder $Z_\mathfrak{X}(\alpha\bar G, F_G)$  of $\mathcal{C}'$. 

 Define a cover $\mathcal{C}''$ from $\mathcal{C}'$ as follows: For each $(\alpha\h\h\h...)\in\mathfrak{X}$, replace all of its basic neighborhoods in $\mathcal{C}'$,  $\{Z_\mathfrak{X}(\alpha\bar H,F_{\alpha H}^\ell)\}_{\ell\in\lambda(\alpha\h)}=\mathfrak{B}_{(\alpha\h\h\h...)}\cap\mathcal{C}'$,  by a single generalized cylinder $Z_\mathfrak{X}(\alpha\bar H, F_{\alpha H})$ such that
$$Z_\mathfrak{X}(\alpha\bar H, F_{\alpha H})=\bigcup_{\ell\in\lambda(\alpha\h)} Z_\mathfrak{X}(\alpha\bar H,F_{\alpha H}^\ell)\qquad\text{(See Remark \ref{rmk:finite_subcover})}.$$

From Lemma \ref{lem:seq_compc<=>finite_follower}, sequential compactness implies that for each $\alpha\in B(\mathfrak{X})\cap B(\A^\N)$ the set $V_\alpha:=\{H\in\V:\ \h\in\F_\mathfrak{X}(\alpha)\}$ is finite. It implies that for each $\alpha$, the family $V_{\alpha,\mathcal{C}''}:=\{H\in\V:\ Z_\mathfrak{X}(\alpha\bar H,F_{\alpha H})\in\mathcal{C}''\}\subset V_\alpha$ is finite. Denote $V_{\alpha,\mathcal{C}''}=\left\{H^1,...,H^{k(\alpha)}\right\}$, and denote
$N_{\alpha,\mathcal{C}''}:=\left\{Z_\mathfrak{X}\left(\alpha\bar H^1,F_{\alpha H^1}\right),...,  
Z_\mathfrak{X}\left(\alpha\bar H^{k(\alpha)},F_{\alpha H^{k(\alpha)}}\right)\right\}$\sloppy, which is the family of all generalized cylinders of the type  \eqref{eq:Type 2 GC} in $\mathcal{C}''$ whose definition uses $\alpha$.

Note that if $\mathcal{C}'$ is finite, then $\mathcal{C}''$ is finite. On the other hand, since from Remark \ref{rmk:finite_subcover} each generalized cylinder  $Z_\mathfrak{X}(\alpha\bar H^i,F_{\alpha H^i})$ in $\mathcal{C}''$ is a finite union of generalized cylinders in  $\{Z_\mathfrak{X}(\alpha\bar H^i,F_{\alpha H^i}^\ell)\}_{\ell\in\lambda(\alpha\h^i)}\subset \mathcal{C}'$, it follows that if $\mathcal{C}''$ is finite, then $\mathcal{C}'$ is finite.

Now, define the cover $\mathcal{\tilde C}$ from $\mathcal{C}''$ as follows: For each $\alpha=\alpha_0...\alpha_{n-1}\in B(\mathfrak{X})\cap B(\A^\N)$,
 such that $ N_{\alpha,\mathcal{C}''}$ is nonempty, define $\tilde N_{\alpha,\mathcal{C}''}$ as the smallest family of disjoint generalized cylinders that refines $ N_{\alpha,\mathcal{C}''}$. Since $N_{\alpha,\mathcal{C}''}$ is finite and for any $\x\in Z_\mathfrak{X}(\alpha\bar H^i,F_{\alpha H^i})\cap Z_\mathfrak{X}(\alpha\bar H^j,F_{\alpha H^j})$ there are only a finite number of possible symbols for $x_n\in H^i\cap H^j$, it follows that $\tilde N_{\alpha,\mathcal{C}''}$ is also finite and can be written as
$$\tilde N_{\alpha,\mathcal{C}''}=\left\{Z_\mathfrak{X}\left(\alpha\bar H^1,G_{\alpha H^1}\right),...,  
Z_\mathfrak{X}\left(\alpha\bar H^{k(\alpha)},G_{\alpha H^{k(\alpha)}}\right),Z_\mathfrak{X}\left(\alpha g^1\right),...,Z_\mathfrak{X}\left(\alpha g^{l(\alpha)}\right)\right\}.$$
Now, $\mathcal{\tilde C}$ is obtained by replacing in $\mathcal{C}''$ the sets of each family $N_{\alpha,\mathcal{C}''}$ by the sets of the family $\tilde N_{\alpha,\mathcal{C}''}$.\\

Observe that,  $\mathcal{\tilde C}$ is also a cover of $\mathfrak{X}$, and $\mathcal{\tilde C}$  is finite if and only if $\mathcal{C}''$ is finite. 
Hence, we conclude by noting that, since the unique nonempty intersections in $\mathcal{C}''$ could occur between sets of a same family  $N_{\alpha,\mathcal{C}''}$, it follows that $\mathcal{\tilde C}$ is a disjoint cover, and so it is a finite cover.

\end{proof}
 
Now, we are able to characterize compactness and local compactness in blur shifts.

\begin{theo}\label{theo:compactness} A blur shift  $\Sigma_\Lambda$ is compact if and only if  $\V_\Lambda$ is a finite family of sets which covers all except a finite number of elements  of $B_1(\Lambda)$.
\end{theo}

\begin{proof}
Suppose $\Sigma_\Lambda$ compact, then from Lemma \ref{lem:seq_compc<=>finite_follower} and Theorem \ref{theo:compact_seq-compact}, it follows that for all $\alpha\in B(\Lambda)$ the sets $V_\alpha:=\{H\in\V_\Lambda:\ \h\in\F_{\Sigma_\Lambda}(\alpha)\}$ and  $W_\alpha:=\F_{\Sigma_\Lambda}(\alpha)\setminus \left( \bigcup_{H\in V_\alpha}\bar H\right)=\F_\Lambda(\alpha)\setminus \left( \bigcup_{H\in V_\alpha} H\right)$ are finite. Taking $\alpha=\epsilon$, the empty word, it follows
$\F_\Lambda(\epsilon)=B_1(\Lambda)$, $V_\epsilon=\V_\Lambda$ and $W_\epsilon=B_1(\Lambda)\setminus \left( \bigcup_{H\in \V_\Lambda} H\right)$, which implies that $\V_\Lambda$ is a finite family which covers all the set of $B_1(\Lambda)$ but the finite set of symbols $W_\epsilon$.\\

On the other hand, if  $\V_\Lambda$ is a finite family of sets which covers all except a finite number of elements  of $B_1(\Lambda)$, then any follower set in $\Sigma_\Lambda$ can be written as the union of a finite number of blurred sets with a finite set of symbols. Hence, since $\Sigma_\Lambda$ is closed, from Lemma \ref{lem:seq_compc<=>finite_follower} we conclude that it is sequentially compact and so compact (Theorem \ref{theo:compact_seq-compact}).

\end{proof}

\begin{prop}\label{prop:compact_first-count} A blur shift $\Sigma_\Lambda$ over an uncountable alphabet cannot be simultaneously compact and first countable.
\end{prop}

\begin{proof}

Just note that if $B_1(\Lambda)$ is uncountable, it is not possible that all three following statement hold: $\V_\Lambda$ is finite;
$\V_\Lambda$ covers all but finite number of symbols of $B_1(\Lambda)$; For all $H\in\V_\Lambda$ we have $H\cap B_1(\Lambda)$ countable.
Hence, from Proposition \ref{prop:first_countable} and Theorem \ref{theo:compactness} we get that either $\Sigma_\Lambda$ is compact or it is first countable.

\end{proof}

As a direct consequence of the previous proposition, we have that:\\

\begin{cor}\label{cor:compact_matrizable} A blur shift $\Sigma_\Lambda$ over an uncountable alphabet cannot be simultaneously compact and metrizable.
\end{cor}

\qed

Next, we characterize locally-compact blur shift spaces.

\begin{theo}\label{theo:local_compactness} A blur shift $\Sigma_\Lambda$ is locally compact if and only if for all $w\in B_1(\Lambda)$ there exists a family of words $\{\vv_\ell\}_{\ell\in\lambda}\subset B(\Lambda)$ such that $\bigcup_{\ell\in\lambda}Z_\Lambda(w\vv_\ell)\cap\Lambda=Z_\Lambda(w)\cap\Lambda$, and for all $\ell\in\lambda$  and $\uu\in B(\Lambda)$ such that $w\vv_\ell\uu\in B(\Lambda)$ there is a finite number of sets in $\V_\Lambda$  that cover all except a finite number of elements of $\F_\Lambda(w\vv_\ell\uu)$.
\end{theo}

\begin{proof} \phantom\\
\begin{description}
\item[$(\Longrightarrow)$] Suppose that $\Sigma_\Lambda$ is locally compact. Given $w_0\in B_1(\Lambda)$, define $\lambda:=Z_\Lambda(w_0)\cap\Lambda$ and for each $\x\in \lambda$ let $\vv_\x=(x_1... x_{k(\x)})\in B(\Lambda)$ such that $\mathfrak{X}_\x:=Z_\Lambda(w_0\vv_{\x})=Z_\Lambda(w_0x_1...x_{k(\x)})$ is compact. From Lemma \ref{lem:seq_compc<=>finite_follower} and Theorem \ref{theo:compact_seq-compact} it follows that for any $\alpha=w_0\vv_\x\uu\in B(\mathfrak{X}_\x)$ we have a finite number of blurred sets that cover all except a finite number of letters of $\F_{\mathfrak{X}_\x}(w_0\vv_\x\uu)\supset \F_\Lambda(w_0\vv_\x\uu)$. Hence, the family $\{\vv_\x\}_{\x\in\lambda}$ satisfies the desired property.

\item[$(\Longleftarrow)$] Given $\x=(x_i)_{i\in\N}\in\Lambda$, consider $w=x_0$ and let $\{\vv_\ell\}_{\ell\in\lambda}\subset B(\Lambda)$ such that $\bigcup_{\ell\in\lambda}Z_\Lambda(x_0\vv_\ell)\cap\Lambda=Z_\Lambda(x_0)\cap\Lambda$ be the family that verifies the property given in the statement of the theorem. Then, there is $\vv_t$ such that $\x\in Z_\Lambda(x_0\vv_t)$, which means, $\vv_t=x_1...x_{N(t)}$ for some $N(t)\geq 0$. By hypothesis, there are a finite number of sets in $\V_\Lambda$  that cover all except a finite number of elements of $\F_\Lambda(w\vv_t\uu)=\F_\Lambda(x_0...x_{N(t)}\uu)$.

Denote $\mathfrak{X}:=Z_\Lambda(x_0...x_{N(t)})$. It follows that $\mathfrak{X}$ is closed and any $\alpha\in B(\mathfrak{X})\cap B(\A^\N)$ is written as either $\alpha=x_0...x_k$ with $k < N(t)$ or as $\alpha=x_0...x_{N(t)}\uu$ for some $\uu\in B(\Lambda)$. In the former case $\F_\mathfrak{X}(x_0...x_k)=\{x_{k+1}\}$, while in the later case the sets $V_\alpha$ and  $W_\alpha$ given in Lemma \ref{lem:seq_compc<=>finite_follower} are also finite. In both cases, from Lemma \ref{lem:seq_compc<=>finite_follower} and Theorem \ref{theo:compact_seq-compact} we conclude that $\mathfrak{X}$ is compact.
\end{description}

\end{proof}

\begin{ex} Let $\A:=\mathbb{R}^+=[0,\infty)$ and as in Example \ref{ex:blurshift_uncount_no_intersc}, for each $\lambda\in[0,1)$ define the set $H_\lambda:=\{x\in\mathbb{R}^+:\ x:=\lambda+k,\ k\in\N\}$, and then define $\V:=\{H_\lambda\}_{\lambda\in[0,1)}$. 

Let $f:\mathbb{R}^+\to [0,1)$ be the function given by 
$$f(\lambda):=\left
\{\frac{\lambda-\lfloor\lambda\rfloor}{k}+k:\ 0<k\leq \lceil\lambda\rceil^\star\right\},$$
where $\lfloor\lambda\rfloor$ denotes the largest integer less than or equal to $\lambda$, and 
$\lceil\lambda\rceil^\star$ denotes the smallest integer strictly greater than $\lambda$.

Hence, the blur shift $\Sigma_\Lambda\subset\Sigma_{\A^\N}^\V$, where $\Lambda$ is defined as follows: 
$$\x\in\Lambda\quad \Longleftrightarrow\quad \forall i\in\N,\ x_{i+1}\in H_{\ell_i}\text{ for some }\ell_i\in f(x_i),$$
is a locally compact blur shift which is not compact.
\end{ex}

Recall that for any blur shift $\Sigma_\Lambda^\V$ there exist directed labeled graphs $\G$ and $\bar \G$ such that $\G\subset \bar\G$, $\Lambda=\Lambda_\G$ and  $\Sigma_\Lambda^\V=\Lambda_{\bar \G}$ (Theorem \ref{theo:blurgraph}). 
 Thus, the local compactness condition given in Theorem \ref{theo:local_compactness}, translated for the graph language, means that, independently of the initial vertex,  any walk on $\G$ will be eventually enclosed in a subgraph where each vertex has all but a finite number of outgoing labels covered by some finite number of blurred sets. We notice that this condition is more general than the RFUM condition in \cite[Proposition 3.12]{GRISU} for  Gonçalves-Royer ultragraph shifts. Indeed, RFUM condition was a sufficient but not necessary condition for local compactness, since it supposes that the whole graph has the property that almost all outgoing edges of each vertex are covered by a finite number of blurred sets. In our general context, RFUM condition corresponds to the condition given in Corollary \ref{cor:local_compactness} below.

\begin{cor}\label{cor:local_compactness} If for any nonempty letter $a\in B_1(\Lambda)$ and $\uu\in B(\Lambda)$ there are a finite number of sets in $\V_\Lambda$ that cover all except a finite number of elements of $\F_\Lambda(a\uu)$, then $\Sigma_\Lambda$ is locally compact. If the previous property also holds for the empty word $\epsilon$, then $\Sigma_\Lambda$ is compact.
\end{cor}

\begin{proof} Note that for each nonempty word $\w=w_0...w_n\in B(\Lambda)$ we have $\F_\Lambda(\w)\subset\F_\Lambda(w_n)$. Hence, the first part of the corollary corresponds to the special case in Theorem \ref{theo:local_compactness} where, for each $w\in B_1(\Lambda)$, the correspondent family  $\{\vv_\ell\}_{\ell\in\lambda}$ is such that $\vv_\ell=\epsilon$ for all $\ell\in\lambda$. \\

If for the empty word $\epsilon$ we have a finite number of sets of $\V_\Lambda$ covering all except a finite number of elements of $\F_\Lambda(\epsilon\uu)$, then, in particular, a finite number of sets of $\V_\Lambda$ cover all except a finite number of elements of $\F_\Lambda(\epsilon)=B_1(\Lambda)$, and from Theorem \ref{theo:compactness} we conclude that $\Sigma_\Lambda$ is compact.

\end{proof}

\section{Shift commuting maps, continuity and generalized sliding block codes}\label{sec:GSBC_for_blur_shift}

In this section we present several results that characterize shift-commuting maps, with particular interest on the case where they are continuous or generalized sliding block codes. The results presented here hold for general blur shifts and they are expressed in the general framework, but their proofs share ideas with those given in 
 \cite{GS2019} in the particular context of Gonçalves-Royer ultragraph shifts. 

We remark that, differently than in \cite{GSS} where it was considered only Ott-Tomforde-Willis shift spaces over countable alphabets, or than in \cite{GS2019} where the alphabet was always countable, here we do not impose any restriction on the cardinality of the alphabet, which can be uncountable. Furthermore, even the case where the blur shift is not first countable is contemplated in our results.\\

Before starting examining the continuity of general shift-commuting maps, we remark that
as occurs in the particular cases studied in \cite{GRISU,Ott_et_Al2014}, the shift map itself is not, in general, everywhere continuous. It is encapsulated in the next proposition.

\begin{prop}\label{prop:cont_shift_map}
Let $\Sigma_\Lambda\subset \Sigma_{\A^\N}$ be a blur shift space. Then
\begin{enumerate}
\item  The shift map is continuous on $\Sigma_\Lambda\setminus \mathcal{L}_0 $;

\item   The shift map is continuous at $(\h\h\h...)\in \mathcal{L}_0(\Lambda)$ if, and only if,  there exists a finite $F\subset H$ such that  $\F_{\Lambda}(H\setminus F)\subset H$ and for all $g\in H$ the set $H\cap\Pp_\Lambda(g)$ is finite.
\end{enumerate}

\end{prop}

\begin{proof}\phantom\\
\begin{enumerate}
\item Let $\x=(x_0x_1x_2...)\in\Sigma_\Lambda\setminus\mathcal{L}_0$, and let $\y:=(y_0y_1...)=\s(\x)=(x_1x_2...)$. Recall that, since $\x\notin\mathcal{L}_0$, we have $x_0\in\A$.
We need to prove that for any given generalized cylinder $U$ containing $\y$, there exists a generalized cylinder $V$ containing $\x$ such that $\s(V)\subset U$. If $\y\in \Lambda$, then it is direct that given any $U:=Z_\Lambda(y_0...y_{k-1})=Z_\Lambda(x_1...x_k)$ it follows $V:=Z_\Lambda(x_0x_1...x_k)\ni\x$ is such that $\s(V)\subset U$. On the other hand, if $\y\in\mathcal{L}_n$ for some $n\in\N$, we have that a generalized cylinder containing $\y$ is in the form $U:=Z_\Lambda(y_0...y_{n-1}\bar H,F)=Z_\Lambda(x_1...x_n\bar H,F)$, and hence 
$V:=Z_\Lambda(x_0x_1...x_n\bar H,F)$ contains $\x$ and is such that  $\s(V)\subset U$.

\item The continuity of the shift map at some point $\x=(\h\h\h...)$ means that given a generalized cylinder neighborhood of $\s(\x)=\x$, say $U:=Z_\Lambda(\bar H,F')$ with $F'\subset H$ finite, there exists a generalized cylinder neighborhood of $\x$, say $V:=Z_\Lambda(\bar H,F)$ with $F\subset H$ finite, such that $\s(V)\subset U$. Note that $\s(V)\subset U$ means that $\F_\Lambda(H\setminus F)\subset H\setminus F'$.

Thus, setting $U:=Z_\Lambda(\bar H)$, the continuity of $\s$ at $\x$ means that $\F_{\Lambda}(H\setminus F)\subset H$ for some finite $F\subset H$. On the other hand, setting $U:=Z_\Lambda(\bar H,F')$ for a nonempty finite $F'\subset H$, the continuity of $\s$ at $\x$ means that there is a finite $F\subset H$ such that $\F_\Lambda(H\setminus F)\subset H\setminus F'$, which in its turn means that for each $g\in F'$ we have $\Pp_\Lambda(g)\cap H$ finite. Since $F'$ can be taken as any finite subset of $H$, it means that for any $g\in H$ we shall have  $\Pp_\Lambda(g)\cap H$ finite.
\end{enumerate}

\end{proof}

Note that condition {\em ii.}  in Proposition \ref{prop:cont_shift_map} above is just ensuring that if $(\x^n)_{n\in\N}\in\Sigma_\Lambda$ is such that $x^n_0\to \h$ as $n\to\infty$, then  $x^n_1\to \h$ as $n\to\infty$ as well. 

\begin{ex} Consider the alphabet $\A:=\N^2$ and  $\V:=\{H_m\}_{m\in\N}$ the family of  blurred sets where $H_m:=\{(m,k):\ k\in\N\}$ defined for each $m\in\N$.

Let $\Lambda\subset\A^\N$ be the shift where $(x_i)_{i\in\N}=(m_i,k_i)_{i\in\N}\in\Lambda$ if and only if for all $i\in\N$ we have $(m_{i+1},k_{i+1})$ is any whenever $k_i\leq m_i$, and $(m_{i+1},k_{i+1})$ is such that $m_{i+1}=m_i$ and $k_{i+1}\geq k_i$ whenever $k_i>m_i$.

 It follows that $\Sigma_\Lambda^\V$ satisfies the hypotheses in Proposition \ref{prop:cont_shift_map}.ii. and then the shift map is continuous on the whole $\Sigma_\Lambda^\V$. In fact, for all $m\in\N$ we have $\F_\Lambda(H_m\setminus\{(m,0),...,(m,m)\})=\{(m,k):\ k>m\}\subset H_m$, and for any $(m,\ell)\in H_m$ it follows that $\Pp_\Lambda\big((m,\ell)\big)\cap H_m=\big\{(m,k):\ k\leq\max\{m,\ell\}\big\}$.

\end{ex}

\subsection{Finitely defined sets}

In this subsection we recall the concepts of pseudo cylinders and finitely defined sets which were introduced in \cite{GSS,GSS1} and developed in \cite{GS2019}.\\

\begin{defn}\label{defn:pseudo_cylinder} A {\bf pseudo cylinder} in a blur shift space $\Sigma_\Lambda$ is a set of the form
$$[\w]_{k}^\ell:=\{(x_i)_{i\in\N}\in\Sigma_\Lambda: (x_{k}\ldots x_\ell)=\w\},$$
where $0\leq k\leq\ell$ and $w\in B_{\ell-k+1}(\Sigma_\Lambda)$. We also assume that the empty set is a pseudo cylinder.
\end{defn}

We recall that the topology of generalized cylinders when restricted to $\Lambda$ coincides with the product topology of $\A^\N$ restricted to $\Lambda$, which means that a pseudo cylinder $[\w]_k^\ell$ with $\w\in B(\Lambda)$ is always an open set of $\Lambda$ and so of $\Sigma_\Lambda$. However, it will be a closed set of  $\Sigma_\Lambda$ if and only if $k=0$ in its definition, in which case $[\w]_{0}^\ell=Z_\Lambda(\w)$. On the other hand, a pseudo cylinder $[\w]_{k}^\ell$, with $\w$ containing a symbol of $\tilde V$, is never an open set of $\Sigma_\Lambda$.

\begin{defn}\label{defn:finitely_defined_set} We say that $C\subset \Sigma_\Lambda$ is {\bf finitely defined} of $\Sigma_\Lambda$ if both $C$ and $C^c$ can be written as unions of pseudo cylinders.
\end{defn}

Clearly from the above definition, if $C$ is a finitely defined set, then $C^c$ is too. 
Intuitively, a set $C$ is finitely defined in $\Sigma_\Lambda$ if we can check whether or not any given $\x\in  \Sigma_\Lambda$ belongs to $C$ by knowing a finite number of coordinates of $\x$.
The empty set and $\Sigma_\Lambda$ are examples of finitely defined sets.  

The next proposition can be proved using the same approach than in propositions 5 and 6 of \cite{GS2019}.

\begin{prop}\label{prop:cylinders_pseudocylinders} Finite unions and finite intersections of generalized cylinders are finitely defined sets.
\end{prop}

\qed

We remark that, in general, infinite unions or intersections of finitely defined sets are not finitely defined sets. Hence, we get that an infinite union of generalized cylinders might not be a finitely defined set.

\subsection{Shift-commuting maps}

In this subsection we shall present a characterization of shift-commuting maps.
In what follows, let $\A$  and $\mathcal{B}$ be two alphabets, and $\V\subset 2^\A$ and  $\U\subset2^\mathcal{B}$ be two families of blurred sets, 
$\Sigma_\Lambda\subset \Sigma_{\A^\N}^\V$ and $\Sigma_\Gamma\subset\Sigma_{\mathcal{B}^\N}^\U$ be two blur shift spaces, and $\Phi : \Sigma_\Lambda \to \Sigma_\Gamma$ be a map. For each $a\in \bar{\mathcal{B}}=\mathcal{B}\cup\tilde\U$ define \begin{equation}\label{Eq:Phi_partition}C_a:=\Phi^{-1}([a]_0^0).\end{equation}
Note that $\{C_a\}_{a\in \bar{\mathcal{B}}}$ is a partition of $\Sigma_\Lambda$.

\begin{prop}\label{prop:shift_invariant_maps}
A map $\Phi:\Sigma_\Lambda\to \Sigma_\Gamma$ is shift commuting (i.e. $\Phi \circ \s = \s \circ \Phi$) if, and only if, for all $\x\in \Sigma_\Lambda$ and $n\geq 0$ we have
    \begin{equation}\label{eq:shift_invariant_maps}\bigl(\Phi(\x)\bigr)_n=\sum_{a\in \bar{\mathcal{B}}}a\mathbf{1}_{C_a}\circ\sigma^{n}(\x), \end{equation}
where $\mathbf{1}_{C_a}$ is the
characteristic function of the set $C_a$ and $\sum$ stands for the symbolic sum.

\end{prop}

\qed

The above result has similar proof than Proposition 3 in \cite{GS2019}, while the next proposition corresponds to Proposition 4 and Corollary 2 in \cite{GS2019}.

\begin{prop}\label{sliding block code->preserves_period}
Let $\Phi:\Sigma_\Lambda\to \Sigma_\Gamma$ be a shift-commuting map. It follows that:

\begin{enumerate}

\item  If $\x\in \Sigma_\Lambda$ is a sequence with period $p\geq 1$ (that is, $\s^p(\x)=\x$) then $\Phi(\x)$ also has period $p$;

\item For all $\tilde G\in\tilde\U$, we have that $\s(C_{\tilde G})\subset C_{\tilde G}$.

\end{enumerate}
 
\end{prop}

\qed

We will say that a shift-commuting map $\Phi:\Sigma_\Lambda\to \Sigma_\Gamma$ is {\bf length preserving} if it maps sequences from $\mathcal{L}_n(\Lambda)$ to $\mathcal{L}_n(\Gamma)$ for all $n\in\N\cup\{\infty\}$. Length-preserving maps were used in \cite{GRISU} and \cite{Ott_et_Al2014} to prove isomorphism between the $C^*$ -algebras of Ott-Tomforde-Willis edge shifts and Gonçalves-Royer ultragraph shifts. The next proposition gives a characterization of length-preserving maps.

\begin{prop}\label{prop:length-preserving_map} A  shift-commuting map $\Phi:\Sigma_\Lambda\to \Sigma_\Gamma$ is length-preserving if and only if $\mathcal{L}_0(\Lambda)=\bigcup_{\tilde G\in\tilde \U}C_{\tilde G}$.
\end{prop}

\begin{proof}

Just note that $\bigcup_{\tilde G\in\tilde \U}C_{\tilde G}=\Phi^{-1}\left(\mathcal{L}_0(\Gamma)\right)$. 
Hence, if $\Phi$ a is length-preserving shift-commuting map, then it is direct that all points of $\mathcal{L}_0(\Lambda)$, and only points of  $\mathcal{L}_0(\Lambda)$, can be mapped to points of $\mathcal{L}_0(\Gamma)$. Conversely, suppose $\Phi$ is not length preserving, and let $\x\in \mathcal{L}_m(\Lambda)$ for some $m\in \N\cup\{\infty\}$, such that $\Phi(\x)\in\mathcal{L}_n(\Gamma)$ with $n\neq m$. Since $\Phi$ is shift commuting, taking $k:=\min\{m,n\}$, we have that $\y:=\s^k(\x)\in \mathcal{L}_{m-k}(\Lambda)$ and $\Phi(\y)=\Phi(\s^k(\x))=\s^k(\Phi(\x))\in \mathcal{L}_{n-k}(\Gamma)$. Thus, either $\y$ is a point of $\mathcal{L}_0(\Lambda)$ which is not mapped to  $\mathcal{L}_0(\Gamma)$ or 
$\y$ is a point not in $\mathcal{L}_0(\Lambda)$ which is mapped to  $\mathcal{L}_0(\Gamma)$. In any case, it follows that $\mathcal{L}_0(\Lambda)\neq\Phi^{-1}\left(\mathcal{L}_0(\Gamma)\right)$.

\end{proof}

The next theorem characterizes general continuous shift-commuting maps (including those which are not generalized sliding block codes - see Definition \ref{defn:GSBC}). Its proof follows analogous outline than Theorem 3.7 in \cite{GS2019}.

\begin{theo}\label{theo:CSC}
If $\Phi:\Sigma_\Lambda\to \Sigma_\Gamma$ is continuous and shift commuting, then:

\begin{enumerate}
\item For each $a\in  \mathcal{B}$, the set $C_a$ is a (possibly empty) union of generalized cylinders of $\Sigma_\Lambda$;

\item If $\x=(x_0...x_{k-1}\h\h\h...)\in \mathcal{L}_k(\Lambda)$ and $\Phi(\x)=(y_0,...y_{\ell-1}\tilde G\tilde G\tilde G...)\in \mathcal{L}_\ell(\Gamma)$, then $\ell\leq k$ and for all $Z_\Gamma(\bar G,F)$ there exists a finite $F'\subset H$ such that $\Phi\big(Z_\Lambda(x_{k-\ell}...x_{k-1}\bar H,F')\big)\subset  Z_\Gamma(\bar G,F)$.

\item If $\x=(\h\h\h...)\in \mathcal{L}_0(\Lambda)$ and $\Phi(\x)=(ddd...)\in\Gamma$, then for all $M>0$ there exists a generalized cylinder $Z_\Lambda(\bar H,F)$ such that $\s^i(Z_\Lambda(\bar H,F)) \subseteq C_d$ for all $i=0, 1, \ldots, M$.

\end{enumerate}

Conversely, if $\Phi$ is continuous on the points of $\Lambda\cap\Phi^{-1}\big(\mathcal{L}_0(\Gamma)\big)$ and condition $i.-iii.$ hold, then $\Phi$ is continuous on the whole $\Sigma_\Lambda$.
\end{theo}

\qed

\subsection{Generalized sliding block codes and continuous shift-commuting maps}

We conclude this section, by stating the analogous of the Curtis-Hedlund-Lyndon theorem in the context of blur shifts.

\begin{defn}\label{defn:GSBC} A map $\Phi:\Sigma_\Lambda\to \Sigma_\Gamma$ is a {\bf generalized sliding block code} if, and only if, for all $\x\in\Sigma_\Lambda$ and $i\in\N$ we have

$$\big(\Phi(\x)\big)_i=\sum_{a\in\tilde\B}a\e_{C_a}\circ\s^i(\x),$$

where each $C_a$ is a finitely defined set.
  
\end{defn}

Generalized sliding block codes were proposed in \cite{SobottkaGoncalves2017} as the natural generalization of the concept of sliding block codes in the context of classical shift spaces over infinite alphabets, where it was proved that the class of generalized sliding block codes coincides with the class of continuous shift-commuting maps. In the contexts of one-sided and two-sided Ott-Tomforde-Willis shifts (see \cite{GSS} and \cite{GSS1}, respectively), and Gonçalves-Royer ultragraph shifts \cite{GS2019}, the class of generalized sliding block codes and the class of continuous shift-commuting maps do not coincide in general, but it is possible to obtain sufficient and necessary conditions under which those classes coincide.

\begin{prop}\label{prop:sliding_bock_codes_continuous_on_infinity_sequences} If $\Phi:\Sigma_\Lambda\to \Sigma_\Gamma$  is a generalized sliding block code then it is continuous on $\Lambda$ and shift commuting. 
\end{prop}

\qed

A proof for the previous proposition can be adapted from \cite[Corollary 3]{GS2019}. In what follows we state sufficient and necessary conditions under which both classes coincide in the general context of blur shifts. The proof of Theorem \ref{theo:general-CHL-T} can be adapted from \cite[Theorem 3.8]{GS2019}.\\

\begin{theo}\label{theo:general-CHL-T}
Suppose that  $\Phi:\Sigma_\Lambda\to \Sigma_\Gamma$ is a map such that for each $\tilde G\in \tilde\U$ the set $C_{\tilde G}$ is a finitely defined set. Then $\Phi$ is continuous and shift commuting if, and only if, $\Phi$ is a generalized sliding block code where:
\begin{enumerate}
\item\label{theo:general-CHL-Theo:condition_1} For any $a\in \mathcal{B}$, the set $C_a$ is a (possibly empty) union of generalized cylinders of $\Sigma_\Lambda$;

\item\label{theo:general-CHL-Theo:condition_2} If $( x_0\ldots  x_{n-1} \h\h\h\ldots)\in \partial\Lambda$  is such that $\Phi( x_0\ldots  x_{n-1} \h\h\h\ldots)=(\tilde G\tilde G\tilde G...\ldots)\in\partial\Gamma$, then:
 \begin{enumerate}[a.]
 \item There exists a finite $F \subset H$ such that $\Phi\big(Z_\Lambda( x_0\ldots  x_{n-1} \bar H,F)\big) \subset Z_\Gamma(\bar G)$;

\item For each $g\in G$, there are just a finite number of $h\in H$ for which there exists some sequence $( x_0\ldots  x_{n-1} h ...)$ in $C_g$.

\end{enumerate}

\item\label{theo:general-CHL-Theo:condition_3} If $\h\in\tilde\V$ is such that $\Phi(\h\h\h\ldots)=(ddd\ldots)\in\Gamma$, then for all $M\geq 1$ there exists a generalized cylinder $Z_\Lambda(\bar H,F)$ such that $\s^i\big(Z_\Lambda(\bar H,F)\big) \subseteq C_d$ for all $i=0, 1, \ldots, M$.

\end{enumerate}

\end{theo}

\qed

The next result can be directly proved by combining Proposition \ref{prop:length-preserving_map} and Theorem \ref{theo:general-CHL-T} above, and corresponds to Corollary 4 in \cite{GS2019}.

\begin{cor}\label{cor:length-preserving_GSBC} A map $\Phi:\Sigma_\Lambda\to \Sigma_\Gamma$ is continuous, shift commuting, and length-preserving, if and only if it is a generalized sliding block code such that:
\begin{enumerate}
\item\label{cor:length-preserving_SBC:condition_1} For each $a\in \mathcal{B}$, the set $C_a$ is a (possibly empty) union of generalized cylinders of $\Sigma_\Lambda$;

\item\label{cor:length-preserving_SBC:condition_2} $\mathcal{L}_0(\Lambda)=\bigcup_{\tilde G\in\tilde \U}C_{\tilde G}$.

\item\label{cor:length-preserving_SBC:condition_3} If $\Phi(\h\h\h\ldots) = (\tilde G\tilde G \tilde G\ldots) \in \partial\Gamma$ then:
 \begin{enumerate}[a.]
 \item There exists a finite $F \subset H$ such that $\Phi\big(Z_\Lambda(\bar H,F)\big) \subset Z_\Gamma(\bar G)$;

\item For each $g\in G$, there are only a finite number of $h\in H$ for which there exists some sequence $( h x_1x_2 ...)$ belonging to $C_g$.

\end{enumerate}

\end{enumerate}
\end{cor}

\qed

\section{Conclusion and open problems}

In this work we proposed and studied a new class of symbolic systems, named blur shifts. Roughly, a blur shift is the closure of a classical shift space with respect to some topology which depends on a chosen resolution $\V$ (any family of infinite subsets of the original alphabet with finite pairwise intersection).  The class of blur shifts includes all classical shift spaces (if resolution $\V=\emptyset$ is taken) and also two recently proposed new classes of symbolic systems: Ott-Tomforde-Willis shifts \cite{Ott_et_Al2014} (if resolution $\V=\{\A\}$ is taken, where $\A$ is the entire original alphabet); and Gonçalves-Royer ultragraph shifts \cite{GRISU} (if resolution $\V$ is taken as the family of all minimal infinite emitters of an ultragraph).

In this section we shall highlight some open problems and some possible applications regarding blur shifts.\\

\begin{prob}\label{prob:metrizability_conditions} To find a complete set of sufficient and necessary conditions for a blur shift to be metrizable. -- Note that we only have a complete characterization of metrizability for compact blur shifts (Corollary \ref{cor:equiv_count_metr}). For the general case of non-compact blur shifts we have found only some sufficient conditions for metrizability  (Theorem \ref{theo:Metrizable}), and we know no other necessary condition than being first countable (which is characterized by Proposition \ref{prop:first_countable}).
\end{prob}

\begin{prob}\label{prob:metric_construction} To construct metrics for non-second countable blur shift spaces. -- Although from Theorem \ref{theo:Metrizable} we have that non-second countable blur shift spaces can be metrizable (see Example \ref{ex:metrizable_blur_shifts}), we have not a specific description of such metrics. In particular, it would be interesting to find general procedures to construct metrics for any non-second countable blur shifts as made in Section \ref{subsec:metric} for second countable blur shifts.
\end{prob}

From the dynamical perspective, it would be interesting to study how the choice of a resolution affects informational and ergodic properties of the shifts. For instance, we could consider the following problem:

\begin{prob}\label{prob:chaos} To study the chaotic behaviour of blur shifts for distinct resolutions. -- For Gonçalves-Royer ultragraph shifts it was found very interesting relations between different types of chaos \cite{GOUG2020-1,GOUG2020-2}. Besides to study those relations for non-Markovian blur shifts, it would also be interesting to check how chaos is affected when we fix a classical shift space (a Markov shift, for instance) and vary the resolution of the blur shift.
\end{prob}

The main application of blur shifts could be to serve as a tool to extend the results about $C^*$-algebra isomorphisms given in \cite{GRISU} and \cite{Ott_et_Al2014} to a wider class of labelled graphs. In order to achieve this kind of result, the first question to be answered concerns on the resolution to be used in the blur shifts:

\begin{prob}\label{prob:natural_resolution} Given a classical shift space, is there some `natural' resolution compatible with the dynamical and algebraic structures? --  Gonçalves-Royer ultragraph shifts were defined as blur shifts whose resolutions depend on the structure of the given ultragraphs (the minimal infinite emitter sets). This leads us to suppose that for each given shift space could there be some `natural' resolution which one should consider when studying the relation between the conjugacy of the shift spaces and the isomorphism of the graph $C^*$-algebra.
\end{prob}

In particular, we expect that blur shifts might be used to study the correspondence between the conjugacy of a subclass of weakly sofic shifts (see sections 6 and 7 of \cite{Sobottka2020}) and the isomorphism of the $C^*$-algebras of their respective labeled graphs. For this specific case of weakly sofic shifts, it is not so naive to think that the `natural' resolution exists and corresponds to some construction analogous to that made for Gonçalves-Royer ultragraph shifts. By conjecturing that this is exactly the `natural' resolution for a weakly sofic shift, and by recalling the RFUM condition (expressed in the statement of Corollary \ref{cor:local_compactness}) and the conditions for a map to be a length-preserving topological conjugacy (given in Corollary \ref{cor:length-preserving_GSBC}), we conjecture as follows.

\begin{conj} Let $\Lambda\subset\A^\N$ and $\Gamma\subset\B^\N$ be two  weakly sofic shifts whose associated labeled graphs are left-resolving and such that there are only finitely many vertexes that are source of each fixed label. Let $\Sigma_\Lambda$ and $\Sigma_\Gamma$ be the respective blur shifts for the `natural' resolutions. Suppose that $\Sigma_\Lambda$ and $\Sigma_\Gamma$ hold the condition RFUM. If $\Sigma_\Lambda$ and $\Sigma_\Gamma$ are  topologically conjugate via a lenght-preserving generalized sliding block code, then the $C^*$-algebras associated to the labeled graphs of $\Lambda$ and $\Gamma$ are isomorphic.
\end{conj}

As first step to prove the above conjecture, we could consider less general versions of it by restricting our study to sofic shifts \cite[Section 6]{Sobottka2020} or by considering (weakly) sofic shifts whose associated labeled graphs are such that there are only finitely many edges with a same label (besides being left-resolving).\\

Furthermore, we expect that blur shifts could be used to extend other results on $C^*$-algebras. For instance, blur shifts could be useful to extend the results about KMS and ground states for a time evolution on the $C^*$-algebra of a graph or ultragraph  \cite{CarlsenLarsen2016,CastroGoncalves2018} to other types of labeled graphs.

%=======================================================================================================================

\section*{Acknowledgments}

\noindent This study was financed in part by the Coordenação de Aperfeiçoamento de Pessoal de Nível Superior - Brasil (CAPES) - Finance Code 001.

\noindent T. Z. Almeida was partially supported by CAPES-Brazil fellowship.

\noindent M. Sobottka was supported by CNPq-Brazil grant 301445/2018-4.
Part of this work was carried out while the author was visiting Professor fellow of CAPES-Brazil at Pacific Institute for the Mathematical Sciences, University of British Columbia.

\noindent The authors thank A. T. Baraviera, U. B. Darji, D. Gonçalves and A. O. Lopes   for valuable suggestions and comments regarding the manuscript.

%====================================================== BIBLIOGRAFIA =================================================================


\begin{thebibliography}{20}


\bibitem{Arkhangelskii _Pontryagin}
{\sc Arkhangel'ski\v i, A. V.} {\sc and} {\sc Pontryagin, L. S.} (1990).
\newblock ``General Topology I '',
\newblock {\em Springer-Verlag, New York.}


\bibitem{CarlsenLarsen2016} 
{\sc Carlsen, T. M.} {\sc and} {\sc Larsen, N. S.}  (2016).
\newblock {\em Partial actions and KMS states on relative graph C*-algebras},
\newblock J. Funct. Anal., {\bf 271}, 8, 2090-2132.

\bibitem{CastroGoncalves2018} 
{\sc Carlsen, T. M.} {\sc and} {\sc Larsen, N. S.}  (2018).
\newblock {\em KMS and Ground States on Ultragraph C*-Algebras},
\newblock Integr. Equ. Oper. Theory, {\bf 2018}, 90:63.

\bibitem{Franklin64}
{\sc Franklin, S.} (1965).
\newblock {\em Spaces in which sequences suffice},
\newblock Fundamenta Mathematicae, {\bf 57}, q, 107-115.

\bibitem{GRISU} 
{\sc Gonçalves, D.} {\sc and} {\sc Royer, D.}  (2019).
\newblock {\em Infinite alphabet edge shift spaces via ultragraphs and their C*-algebras},
\newblock Int. Math. Res. Not., {\bf 2019}, 2177-2203.

%\bibitem{GRultra}
%{\sc Gonçalves, D.} {\sc and} {\sc Royer, D.}(2017) {\em Ultragraphs and shift spaces over infinite alphabets}, Bull. Sci. Math., {\bf 141}, 25-45.

\bibitem{GR}
{\sc Gonçalves, D.} {\sc and} {\sc Royer, D.}  (2015).
\newblock  {\em (M + 1)-step shift spaces that are not conjugate to M-step shift spaces},
\newblock Bulletin des Sciences Math\'{e}matiques (Paris. 1885), \textbf{139}, issue 2, 178-183.

\bibitem{GS2019} {\sc Gonçalves, D.} {\sc and} {\sc Sobottka, M.} (2019).
\newblock {\em Continuous shift commuting maps between ultragraph shift spaces.},
\newblock  Discrete and Continuous Dynamical Systems, {\bf 39} (2), 1033-1048.


\bibitem{GSS} {\sc Gonçalves, D.}, {\sc Sobottka, M.} {\sc and} {\sc Starling, C.} (2016).
\newblock {\em Sliding block codes between shift spaces over infinite alphabets},
\newblock Math. Nachr., {\bf 289} (17-18), 2178-2191.

\bibitem{GSS1} {\sc Gonçalves, D.}, {\sc Sobottka, M.} {\sc and} {\sc Starling, C.} (2017).
\newblock {\em Two-sided shift spaces over infinite alphabets},
\newblock J. Aust. Math. Soc., {\bf 103} (3), 357-386.

\bibitem{GOUG2020-1}  {\sc Gonçalves, D.} {\sc and} {\sc Uggioni, B. B.} (2020).
\newblock  {\em Ultragraph shift spaces and chaos}, 
\newblock  Bulletin des Sciences Mathématiques, {\bf 158}, 102807.

\bibitem{GOUG2020-2}  {\sc Gonçalves, D.} {\sc and} {\sc Uggioni, B. B.} (2020).
\newblock  {\em Li-Yorke chaos for ultragraph shift spaces}, 
\newblock  Discrete and Continuous Dynamical Systems, {\bf 40} (4), 2347-2365.



\bibitem{Munkres}
{\sc Munkres, J.~R.} (2000).
\newblock ``Topology'',
\newblock {\em Prentice Hall, Upper Saddle River.}

\bibitem{Sobottka2020}
{\sc Sobottka, M.} (2020).
\newblock {\em Some notes on the classification of shift spaces: Shifts of Finite Type; Sofic shifts; and Finitely Defined Shifts},
\newblock  Preprint on arXiv: 2010.10595.

\bibitem{SobottkaGoncalves2017}
{\sc Sobottka, M.} {\sc and} {\sc Gonçalves, D.} (2017).
\newblock {\em A note on the definition of sliding block codes and the Curtis-Hedlund-Lyndon Theorem},
\newblock  Journal of Cellular Automata, ~\textbf{12},~3--4, 209--215.

\bibitem{Ott_et_Al2014}
{\sc Ott, W.}, {\sc Tomforde, M.} {\sc and} {\sc Willis,
P.~N.} (2014).
\newblock  {\em One-sided shift spaces over infinite alphabets}, New York Journal of Mathematics. NYJM Monographs 5. State University of New York, University at Albany, Albany, NY. 54 pp.




\end{thebibliography}
\end{document}